\documentclass[11pt,a4paper]{article}
\usepackage[utf8]{inputenc}
\usepackage[T1]{fontenc}
\usepackage{amsmath}
\usepackage{amsthm}
\usepackage{amsfonts}
\usepackage{amssymb}
\usepackage{mathrsfs}
\usepackage{url}
\usepackage{mathdots}
\usepackage{geometry}
\usepackage{verbatim}
\usepackage{hyperref}

\hypersetup{
    bookmarks=true,         
    unicode=false,          
    pdftoolbar=true,        
    pdfmenubar=true,        
    pdffitwindow=false,     
    pdfstartview={FitH},    
    pdftitle={Eigenvarieties for classical groups and complex conjugations in Galois representations},    
    pdfauthor={Olivier Taïbi},     
    colorlinks=true,       
    linkcolor=blue,          
    citecolor=green,        
    filecolor=green,      
    urlcolor=cyan}           

\newcommand{\Galg}{\mathbf{G}}
\newcommand{\Talg}{\mathbf{T}}
\newcommand{\Salg}{\mathbf{S}}
\newcommand{\Balg}{\mathbf{B}}
\newcommand{\Nalg}{\mathbf{N}}
\newcommand{\Spalg}{\mathbf{Sp}}
\newcommand{\Hom}{\mathrm{Hom}}

\newcommand{\Q}{\mathbb{Q}}
\newcommand{\A}{\mathbb{A}}
\providecommand{\C}{\mathbb{C}}
\renewcommand{\C}{\mathbb{C}}    
\newcommand{\N}{\mathbb{N}}
\newcommand{\Qp}{\mathbb{Q}_p}
\newcommand{\Qpbar}{\overline{\mathbb{Q}}_p}
\newcommand{\Qbar}{\overline{\mathbb{Q}}}

\newcommand{\F}{\mathbb{F}}

\newcommand{\Zp}{\mathbb{Z}_p}
\newcommand{\Z}{\mathbb{Z}}

\newcommand{\Ocal}{\mathcal{O}}
\newcommand{\Ccal}{\mathcal{C}}
\newcommand{\Xscr}{\mathscr{X}}

\newcommand{\Hcal}{\mathcal{H}}
\newcommand{\Wscr}{\mathscr{W}}

\newcommand{\R}{\mathbb{R}}

\newcommand{\St}{\mathrm{St}}
\newcommand{\Ind}{\mathrm{Ind}}
\newcommand{\Res}{\mathrm{Res}}
\newcommand{\Tr}{\mathrm{Tr}}
\newcommand{\BC}{\mathrm{BC}}

\newcommand{\GL}{\mathrm{GL}}
\newcommand{\SL}{\mathrm{SL}}
\newcommand{\SO}{\mathrm{SO}}
\newcommand{\Sp}{\mathrm{Sp}}

\newcommand{\Fil}{\mathrm{Fil}}

\newcommand{\ilim}{\mathop{\varinjlim}\limits}

\newcommand{\Bcris}{B_{\mathrm{cris}}}

\newcommand{\Dcris}{D_{\mathrm{cris}}}

\newcommand{\Ddr}{D_{\mathrm{dR}}}

\newtheorem{theo}{Theorem}[subsection]
\newtheorem*{theononumb}{Theorem}
\newtheorem{lemm}[theo]{Lemma}
\newtheorem*{conj}{Conjecture}
\newtheorem{assu}[theo]{Assumption}
\newtheorem{coro}[theo]{Corollary}
\newtheorem{defi}[theo]{Definition}
\newtheorem{prop}[theo]{Proposition}
\newtheorem{rema}[theo]{Remark}

\newtheorem{theointro}{Theorem}

\numberwithin{equation}{subsection}

\begin{document}

\baselineskip=16pt

\author{Olivier Taïbi}
\title{Eigenvarieties for classical groups and complex conjugations in Galois representations}
\date{}

\maketitle

\begin{abstract}
The goal of this paper is to remove the irreducibility hypothesis in a theorem of Richard Taylor describing the image of complex conjugations by $p$-adic Galois representations associated with regular, algebraic, essentially self-dual, cuspidal automorphic representations of $\GL_{2n+1}$ over a totally real number field $F$.
We also extend it to the case of representations of $\GL_{2n}/F$ whose multiplicative character is ``odd''.
We use a $p$-adic deformation argument, more precisely we prove that on the eigenvarieties for symplectic and even orthogonal groups, there are ``many'' points corresponding to (quasi-)irreducible Galois representations.
The recent work of James Arthur describing the automorphic spectrum for these groups is used to define these Galois representations, and also to transfer self-dual automorphic representations of the general linear group to these classical groups.
\end{abstract}

\tableofcontents
\newpage

\section{Introduction}

Let $p$ be a prime.
Let us choose once and for all algebraic closures $\overline{\Q},\Qpbar,\C$ and embeddings $\iota_p : \overline{\Q} \hookrightarrow \Qpbar$, $\iota_{\infty} : \overline{\Q} \hookrightarrow \C$.
Let $F$ be a totally real number field.
A regular, L-algebraic, essentially self-dual, cuspidal (RLAESDC) representation of $\GL_n(\A_F)$ is a cuspidal automorphic representation $\pi$ together with an algebraic character $\eta |\cdot|^q$ of $\A_F^{\times}/F^{\times}$ ($\eta$ being an Artin character, and $q$ an integer) such that
\begin{itemize}
\item $\pi^{\vee} \simeq \eta |\det|^q \otimes \pi$,
\item $\eta_v(-1)$ does not depend on the real place $v$ of $F$ (the common value will be denoted $\eta_{\infty}(-1)$),
\item For any real place $v$ of $F$, $\mathcal{LL}(\pi_v)|_{\mathrm{W}_{\C}} \simeq \oplus_i \left( z \mapsto z^{a_{v,i}} \bar{z}^{b_{v,i}} \right)$ where $\mathcal{LL}$ is the local Langlands correspondence, $\mathrm{W}_{\C} \simeq \C^{\times}$ is the Weil group of $\C$, and $a_{v,i}$, $b_{v,i}$ are integers and $a_{v,i} \neq a_{v,j}$ if $i \neq j$.
\end{itemize}
By definition, $\pi$ is regular, L-algebraic, essentially self-dual, cuspidal (RLAESDC) if and only if $\pi \otimes |\det|^{(n-1)/2}$ is regular, \emph{algebraic} (in the sense of Clozel), essentially self-dual, cuspidal (RAESDC).
The latter is the notion of ``algebraic'' usually found in the literature, and is called ``C-algebraic'' in \cite{BuzGee}.
Given a RLAESDC representation $\pi$ of $\GL_n(\A_F)$, there is (Theorem \ref{theo:galrepgl}) a unique continuous, semisimple Galois representation $\rho_{\iota_p,\iota_{\infty}}(\pi) : G_F \rightarrow \GL_n(\Qpbar)$ such that $\rho_{\iota_p,\iota_{\infty}}(\pi)$ is unramified at any finite place $v$ of $F$ not lying above $p$ for which $\pi_v$ is unramified, and $\iota_{\infty} \iota_p^{-1} \Tr\left(\rho_{\iota_p,\iota_{\infty}}(\pi)(\mathrm{Frob}_v)\right)$ is equal to the trace of the Satake parameter of $\pi_v$ (contained in this assertion is the fact that this trace is algebraic over $\Q$). 
It is conjectured that for any real place $v$ of $F$, if $c_v \in G_F$ is the conjugacy class of complex conjugations associated with $v$, the conjugacy class of $\rho_{\iota_p,\iota_{\infty}}(\pi)(c_v)$ is determined by $\mathcal{LL}(\pi_v)$ (see \cite{BuzGee}[Lemma 2.3.2] for the case of an arbitrary reductive group).
In the present case, by Clozel's purity lemma and by regularity, $\mathcal{LL}(\pi_v)$ is completely determined by its restriction to $\mathrm{W}_{\C}$, and since $\det\left( \rho_{\iota_p,\iota_{\infty}}(\pi)\right)$ is known, the determination of $\rho_{\iota_p,\iota_{\infty}}(\pi)(c_v)$ amounts to the following
\begin{conj}
Under the above hypotheses, $\left| \Tr\left(\rho_{\iota_p,\iota_{\infty}}(\pi)(c_v) \right) \right| \leq 1$.
\end{conj}
There are several cases for which this is known.
When $\eta_{\infty}(-1) (-1)^q = -1$ (this happens only if $n$ is even, and by \cite{BC} this means that $\rho_{\iota_p,\iota_{\infty}}(\pi)$ together with the character $\rho_{\iota_p,\iota_{\infty}}(\eta |\cdot|^q)=(\eta \circ \mathrm{rec}) \mathrm{cyclo}^q$, is ``symplectic''), $\rho_{\iota_p,\iota_{\infty}}(\pi)(c_v)$ is conjugate to $-\rho_{\iota_p,\iota_{\infty}}(\pi)(c_v)$, so the trace is obviously zero.

In \cite{Tay}, Richard Taylor proves the following
\begin{theononumb}[Taylor]
Let $F$ be a totally real number field, $n \geq 1$ an integer.
Let $\pi$ be a regular, L-algebraic, essentially self-dual, cuspidal automorphic representation of $\GL_{2n+1}/F$.
Assume that the attached Galois representation $\rho_{\iota_p,\iota_{\infty}}(\pi) : G_F \rightarrow \GL_{2n+1}(\Qpbar)$ is irreducible.
Then for any real place $v$ of $F$,
\[ \Tr \left( r_{\iota_p,\iota_{\infty}}(\pi)(c_v) \right) = \pm 1 . \]
\end{theononumb}

Although one expects $\rho_{\iota_p,\iota_{\infty}}(\pi)$ to be always irreducible, this is not known in general.
However it is known when $n \leq 5$ by \cite{CalGee}, and for a set of primes $p$ of density one when $n$ is arbitrary but the weights are ``extremely regular'' by \cite[Theorem D]{BLGGT}.

In this paper, the following cases are proved:

\begin{theointro}[Theorem \ref{theo:lastthm}]
\label{theo:thma}
Let $n \geq 2$, $F$ a totally real number field, $\pi$ a regular, L-algebraic, essentially self-dual, cuspidal representation of $\GL_n(\A_F)$, such that $\pi^{\vee} \simeq \left((\eta |\cdot|^q) \circ \det\right) \otimes \pi$, where $\eta$ is an Artin character and $q$ an integer.
Suppose that one of the following conditions holds
\begin{enumerate}
\item $n$ is odd.
\item $n$ is even, $q$ is even, and $\eta_{\infty}(-1)=1$.
\end{enumerate}
Then for any complex conjugation $c \in G_F$, $|\Tr(\rho_{\iota_p,\iota_{\infty}}(\pi)(c))| \leq 1$.
\end{theointro}

This is achieved thanks to the result of Taylor, Arthur's endoscopic transfer between twisted general linear groups and symplectic or orthogonal groups, and using eigenvarieties for these groups.
Let us describe the natural strategy that one might consider to prove the odd-dimensional case using these tools, to explain why it fails and how a detour through the even-dimensional case allows to conclude.

Let $\pi$ be a RLAESDC representation of $\GL_{2n+1}(\A_F)$.
Up to a twist by an algebraic character $\pi$ is self-dual and has trivial central character.
Conjecturally, there should be an associated self-dual Langlands parameter $\phi_{\pi} : \mathrm{L}_F \rightarrow \GL_{2n+1}(\C)$ where $\mathrm{L}_F$ is the conjectural Langlands group.
Up to conjugation, $\phi_{\pi}$ takes values in $\SO_{2n+1}(\C)$, and by functoriality there should be a discrete automorphic representation $\Pi$ of $\Sp_{2n}(\A_F)$ such that $\mathcal{LL}(\Pi_v)$ is equal to $\mathcal{LL}(\pi_v)$ via the inclusion $\SO_{2n+1}(\C)$ for any place of $F$ which is either archimedean or such that $\pi_v$ is unramified.
Arthur's results in his upcoming book \cite{Arthur} imply that this (in fact, much more) holds.
To construct $p$-adic families of automorphic representations (i.e.\ eigenvarieties) containing $\Pi$, it is preferrable to work with a group which is \emph{compact} at the real places of $F$, and work with representations having Iwahori-invariants at the $p$-adic places.
A suitable solvable base change allows to assume that $[F:\Q]$ is even and that $\pi_v$ has Iwahori-invariants for $v | p$.
The last chapter of \cite{Arthur} will allow to ``transfer'' $\pi$ to an automorphic representation $\Pi$ of $\Galg$, the inner form of $\Sp_{2n}$ which is split at the finite places and compact at the real places of $F$.
By \cite{Loe} (which generalizes \cite{TheseG}), the eigenvariety $\Xscr$ for $\Galg$ is available.
Thanks to \cite{Arthur}, one can associate $p$-adic Galois representations $\rho_{\iota_p,\iota_{\infty}}(\cdot)$ to automorphic representations of $\Galg$, yielding a family of Galois representations on $\Xscr$, that is to say a continuous map $T : G_F \rightarrow \Ocal(\Xscr)$ which specializes to $\Tr\left( \rho_{\iota_p,\iota_{\infty}}(\cdot)\right)$ at the points of $\Xscr$ corresponding to automorphic representations of $\Galg(\A_F)$. 
One can then hope to prove a result similar to \cite[Lemma 3.3]{BC}, i.e.\ show that one can ``deform'' $\Pi$ (on $\Xscr$) to reach a point corresponding to an automorphic representation $\Pi'$ whose Galois representation is irreducible (even when restricted to the decomposition group of a $p$-adic place of $F$).
Since $\rho_{\iota_p,\iota_{\infty}}(\Pi')$ comes from an automorphic representation $\pi'$ of $\GL_{2n+1}$, $\pi'$ is necessarily cuspidal and satisfies the hypotheses of Taylor's theorem.
Since $T(c_v)$ is locally constant on $\Xscr$, we would be done.

Unfortunately, it does not appear to be possible to reach a representation $\Pi'$ whose Galois representation is irreducible by using local arguments on the eigenvariety.
However we will prove the following, which includes the case of some even-dimensional special orthogonal groups as it will be needed later:

\begin{theointro}[Theorem \ref{theo:mainressympl}, Theorem \ref{theo:mainresorth}]
\label{theo:thmb}
Let $\Galg$ be an inner form of $\Sp_{2n}$ or $\SO_{4n}$ over a totally real number field, compact at the real places and split at the $p$-adic ones.
Let $\Pi$ be an irreducible automorphic representation of $\Galg(\A_F)$ having Iwahori invariants at all the places of $F$ above $p$, and having invariants under an open subgroup $U$ of $\Galg(\A_{F,f}^{(p)})$.
Let $\rho_{\iota_p,\iota_{\infty}}(\Pi)$ denote the $p$-adic representation of the absolute Galois group $G_F$ of $F$ associated with $\Pi$ and embeddings $\iota_p : \Qbar \hookrightarrow \Qpbar$, $\iota_{\infty} : \Qbar \hookrightarrow \C$.
Let $N$ be an integer.
There exists an automorphic representation $\Pi'$ of $\Galg(\A_F)$ such that:
\begin{itemize}
\item $\Pi'$ is unramified at the places above $p$, and has invariants under $U$;
\item The restriction of $\rho_{\iota_p,\iota_{\infty}}(\Pi')$ to the decomposition group at any place above $p$ is either irreducible or the sum of an Artin character and an irreducible representation of dimension $2n$ (the latter occurring only in the symplectic case);
\item For all $g$ in $G_F$, $\Tr(\rho_{\iota_p,\iota_{\infty}}(\Pi')(g)) \equiv \Tr(\rho_{\iota_p,\iota_{\infty}}(\Pi)(g)) \mod p^N$.
\end{itemize}
\end{theointro}

The possible presence of an Artin character (in the case of inner forms of $\Sp_{2n}$) comes from the fact that the ``standard'' representation of $\SO_{2n+1}(\C)$ in $\GL_{2n+1}(\C)$ is not minuscule: the set of characters of a torus $T(\C)$ of $\SO_{2n+1}(\C)$ in this representation has two orbits under the Weyl group, one of which contains only the trivial character.
The key fact allowing to prove the above theorem is that classical points on the eigenvariety for $\Galg$ correspond to automorphic representations $\Pi$ of $\Galg(\A_F)$ (say, unramified at the $p$-adic places) \emph{and} a refinement of each $\Pi_v$, $v|p$, that is to say a particular element in $T(\C)$ in the conjugacy class of the Satake parameter of $\Pi_v$.
The variation of the crystalline Frobenius of $\rho_{\iota_p,\iota_{\infty}}(\cdot)$ on the eigenvariety with respect to the weight and the freedom to change the refinement (by the action of the Weyl group) are at the heart of the proof of Theorem \ref{theo:thmb}.

Although the strategy outlined above fails, Theorem \ref{theo:thma} can be deduced from Theorem \ref{theo:thmb}.
Indeed the precise description of the discrete automorphic spectrum of symplectic and orthogonal groups by Arthur shows that formal sums of distinct cuspidal self-dual representations of general linear groups ``contribute'' to this spectrum.
The even-dimensional case in Theorem \ref{theo:thma} will be proved by transferring $\pi \boxplus \pi_0$, where $\pi,\pi_0$ are regular, L-algebraic, self-dual, cuspidal representations of $\GL_{2n}(\A_F)$ (resp.\ $\GL_3(\A_F)$) with distinct weights at any real place of $F$, to an automorphic representation $\Pi$ of an inner form $\Galg$ of $\Sp_{2n+2}/F$.
Since $\rho_{\iota_p,\iota_{\infty}}(\pi)\oplus \rho_{\iota_p,\iota_{\infty}}(\pi_0)$ does not contain any Artin character (the zero Hodge-Tate weights come from $\rho_{\iota_p,\iota_{\infty}}(\pi_0)$, which is known to be irreducible), for big enough $N$ any representation $\Pi'$ as in \ref{theo:thmb} has an irreducible Galois representation.

To treat the original case of a regular, L-algebraic, self-dual, cuspidal representation of $\GL_{2n+1}(\A_F)$ having trivial central character, we appeal to Theorem \ref{theo:thmb} for special orthogonal groups.
For example, if $n$ is odd, $\pi \boxplus \pi_0$, where $\pi_0$ is the trivial character of $\A_F^{\times}/F^{\times}$, contributes to the automorphic spectrum of $\Galg$, which is now the special orthogonal group of a quadratic form on $F^{2n+2}$ which is definite at the real places and split at the finite places of $F$.
Note that $\pi \boxplus \pi_0$ is not regular: the zero weight appears twice at each real place of $F$.
However the Langlands parameters of representations of the compact group $\SO_{2n+2}(\R)$ are of the form
\[ \bigoplus_{i=1}^{n+1} \Ind_{W_{\R}}^{W_{\C}} \left( z \mapsto (z/\bar{z})^{k_i} \right) \]
when composed with $\SO_{2n+2}(\C) \hookrightarrow \GL_{2n+2}(\C)$, with $k_1 > \ldots > k_{n+1} \geq 0$.
Moreover $\mathcal{LL}\left((\pi \boxplus \pi_0)_v \right)$ is of the above form, with $k_{n+1}=0$.
The rest of the proof is identical to the odd-dimensional case.

This fact also shows that some \emph{non-regular}, L-algebraic, self-dual, cuspidal representations of $\GL_{2n}(\A_F)$ contribute to the automorphic spectrum of $\Galg$.
Consequently we can also extend Taylor's result to the Galois representations associated with these slightly non-regular automorphic representations.
These Galois representations were shown to exist by Wushi Goldring \cite{Goldring}.

We now fix some notations for the rest of the article.
The valuation $v_p$ of $\Qpbar$ is the one sending $p$ to $1$, and $| \cdot |$ will denote the norm $p^{-v_p(\cdot)}$.
All the number fields in the paper will sit inside $\Qbar$.
We have chosen arbitrary embeddings $\iota_p : \overline{\Q} \hookrightarrow \Qpbar$, $\iota_{\infty} : \overline{\Q} \hookrightarrow \C$.
In fact, the constructions will only depend on the identification between the algebraic closures of $\Q$ in $\Qpbar$ and $\C$ (informally, $\iota_p \iota_{\infty}^{-1}$).
Observe that the choice of a $p$-adic place $v$ of a number field $F$ and of an embedding $F_v \hookrightarrow \Qpbar$ is equivalent, via $\iota_p$, to the choice of an embedding $F \hookrightarrow \Qbar$.
The same holds for the infinite places and $\iota_{\infty}$.
Thus if $F$ is totally real, $\iota_p \iota_{\infty}^{-1}$ defines a bijection between the set of infinite places of $F$ and the set of $p$-adic places $v$ of $F$ together with an embedding $F_v \hookrightarrow \Qpbar$.
The eigenvarieties will be rigid analytic spaces (in the sense of Tate).
If $\Xscr$ is a rigid analytic space over a finite extension $E$ of $\Qp$, $|\Xscr|$ will denote its points.

I would like to thank Gaëtan Chenevier for introducing me to this fascinating subject, Colette Moeglin and Jean-Loup Waldspurger for Lemma \ref{lemm:iwahoritransfersp}, and James Arthur for his help.

\section{Assumptions on the forthcoming chapter of Arthur's book}

As the results of this paper rely on the main theorem of chapter 9 of \cite{Arthur} (the case of inner forms of quasi-split classical groups), which is not yet available, we have stated some properties as assumptions: Assumptions \ref{prop:transfersp}, \ref{lemm:transfergl2sp}, \ref{lemm:transfergl2so} and \ref{lemm:transfernonreg}.
These will all be consequences of the main theorem of \cite[Chapter 9]{Arthur}, very similar to \cite[Theorem 1.5.2]{Arthur}.
In fact the difference will only be of local nature (see \cite{ArthurNote} and \cite{Kottwitz}).
Since the inner forms of classical groups we will consider are split at all the finite places, to state the analog of \cite[Theorem 1.5.2]{Arthur} in these cases one would only need to know how to normalize the local A-packets, i.e.\ fix a Whittaker functional at each finite place and specify a ``strong form'' above the relevant inner form at the archimedean places.
We are unable to guess what global choice could determine these compatible local normalizations in general, except for the case of \emph{pure} inner forms.
This is the case for the orthogonal groups of section \ref{section:orth}, but not for the ``symplectic'' groups of section \ref{section:sympl}.

A subsequent version of this paper will have the assumptions replaced by actual propositions or lemmas.

\section{The eigenvariety for definite symplectic groups}
\label{section:sympl}

In this section we recall the main result of \cite{Loe} in our particular case (existence of the eigenvariety for symplectic groups), and show that the points corresponding to unramified, ``completely refinable'' automorphic forms, with weight far from the walls, are ``dense'' in this eigenvariety.

\subsection{The eigenvariety}

\subsubsection{Symplectic groups compact at the archimedean places}

Let $F$ be a totally real number field of even degree over $\Q$, and let $D$ be a quaternion algebra over $F$, unramified at all the finite places of $F$ ($F_v \otimes_{F} D \simeq \mathrm{M}_2(F_v)$), and definite at all the real places of $F$. Such a $D$ exists thanks to the exact sequence relation the Brauer groups of $F$ and the $F_v$.
Let $n$ be a positive integer, and let $\Galg$ be the algebraic group over $F$ defined by the equation $M^* M = \mathrm{I}_n$ for $M \in \mathrm{M}_n(D)$, where $\left( M^* \right)_{i,j} = M_{j,i}^*$, and $\cdot^*$ denotes conjugation in $D$.

Then $\Galg \left( F \otimes_{\Q} \R \right)$ is a compact Lie group, and for all finite places $v$ of $F$, $\Galg \times_F F_v \simeq \Spalg_{2n} /F_v$.

Fix a prime $p$.
We will apply the results of \cite{Loe} to the group $\Galg'=\mathrm{Res}^F_{\Q} \Galg$.
Let $E$ be a finite and Galois extension of $\Qp$, containing all the $F_v$ ($v$ over $p$).

\subsubsection{The Atkin-Lehner algebra}

The algebraic group $\Galg' \times_{\Q} \Qp = \prod_{v | p} \Galg \times_{\Q} F_v$ (where $v$ runs over the places of $F$) is isomorphic to $\prod_{v | p} \mathrm{Res}_{\Qp}^{F_v} \Spalg_{2n}/F_v$, which is quasi-split but not split in general.
The algebraic group $\Spalg_{2n}$ is defined over $\Z$ by the equation $^tMJM=J$ in $\mathrm{M}_{2n}$, where $J=\begin{pmatrix}
0 & J_n \\ 
-J_n & 0
\end{pmatrix} $ and $J_n = \begin{pmatrix}
0 &  & 1\\
 & \iddots & \\
1 &  & 0
\end{pmatrix}$.
We define its algebraic subgroups $\Talg_v$, $\Balg_v$, $\bar{\Balg}_v$, $\Nalg_v$, $\bar{\Nalg}_v$ of diagonal, upper triangular, lower triangular, unipotent upper triangular, and unipotent lower triangular matrices of $\Res_{\Qp}^{F_v} \Sp_{2n}/F_v$, and let $\Talg = \prod_{v | p} \Talg_v$, $\Balg = \prod_{v | p} \Balg_v$, and so on.
In \cite[2.4]{Loe}, only the action of the maximal split torus of $\Galg' \times_{\Q} \Qp$ is considered.
For our purpose, we will need to extend this and consider the action of a maximal (non-split in general) torus, that is $\Talg$, instead of a maximal split torus $\Salg \subset \Talg$.
The results in \cite{Loe} are easily extended to this bigger torus, essentially because $\Talg(\Qp) \slash \Salg(\Qp)$ is compact.
Moreover, we let $I_v$ be the compact subgroup of $\Spalg_{2n}\left(\Ocal_v\right)$ consisting of matrices with invertible diagonal elements and elements of positive valuation below the diagonal.
Finally, following Loeffler's notation, we let $G_0= \prod_{v | p} I_v$. It is an Iwahori sugroup of $\Galg'(\Qp)$ having an Iwahori decomposition: $G_0 \simeq \bar{N}_0 T_0 N_0$ where $*_0 = *(\Qp) \cap G_0$.

For each place $v$ of $F$ above $p$, let us choose a uniformizer $\varpi_v$ of $F_v$.
Let $\Sigma_v$ be the subgroup of $\Spalg_{2n}(F_v)$ consisting of diagonal matrices whose diagonal elements are powers of $\varpi_v$, i.e.\ matrices of the form
\[\begin{pmatrix}
\varpi_v^{r_1} &  &  &  &  &  \\ 
 & \ddots &  &  &  &  \\ 
 &  & \varpi_v^{r_n} &  &  &  \\ 
 &  &  & \varpi_v^{-r_n} &  &  \\ 
 &  &  &  & \ddots &  \\ 
 &  &  &  &  & \varpi_v^{-r_1}
\end{pmatrix} \]
Let $\Sigma_v^+$ be the submonoid of $\Sigma_v$ whose elements satisfy $r_1 \leq \ldots \leq r_n \leq 0$, and $\Sigma_v^{++}$ the one whose elements satisfy $r_1 < \ldots < r_n < 0$.
Naturally, we set $\Sigma = \prod_{v | p} \Sigma_v$, and similarly for $\Sigma^+$ and $\Sigma^{++}$.

The \emph{Atkin-Lehner algebra} $\mathcal{H}^+_p$ is defined as the subalgebra of the Hecke-Iwahori algebra $\mathcal{H}(G_0 \backslash \Galg'(\Qp) / G_0 )$ (over $\Q$) generated by the characteristic functions $\left[ G_0 u G_0\right]$, for $u \in \Sigma^+$.
Let $\mathcal{H}_p$ be the subalgebra of $\mathcal{H}(G_0 \backslash \Galg'(\Qp) / G_0 )$ generated by the characteristic functions $\left[ G_0 u G_0\right]$ and their inverses, for $u \in \Sigma^+$ (in \cite{IwahoriMatsumoto}, a presentation of the Hecke-Iwahori algebra is given, which shows that $\left[ G_0 u G_0 \right]$ is invertible if $p$ is invertible in the ring of coefficients).

If $S^p$ is a finite set of finite places of $F$ not containing those over $p$, let $\Hcal^S$ be the Hecke algebra (over $\Q$)
\[ \sideset{}{'} \bigotimes_{w \notin S^p \cup S_p \cup S_{\infty}} \Hcal(\Galg(\Ocal_{F_w}) \backslash \Galg(F_w) / \Galg(\Ocal_{F_w})) \]
where $S_*$ denotes the set of places above $*$.
This Hecke algebra has unit $e^S$.
Let $\Hcal_S^p$ be a commutative subalgebra of $\bigotimes_{w \in S^p} \Hcal(\Galg(F_w))$, with unit $e_{S^p}$.

Finally, we let $\Hcal^+ = \Hcal^+_p \otimes \Hcal_{S^p} \otimes \Hcal^S$, $\Hcal = \Hcal_p \otimes \Hcal_{S^p} \otimes \Hcal^S$ and $e = e_{G_0} \otimes e_{S^p} \otimes e^S$.

\subsubsection{\texorpdfstring{$p$}{p}-adic automorphic forms}

The construction in \cite{Loe} depends on the choice of a parabolic subgroup $\mathbf{P}$ of $\Galg'$ and a representation $V$ of a compact subgroup of the Levi quotient $\mathbf{M}$ of $\mathbf{P}$.
The parabolic subgroup we consider here is the Borel subgroup $\Balg$, and thus, using Loeffler's notation, $\Talg=\mathbf{M}$ is a maximal (non-split in general) torus contained in $\Balg$.
The representation $V$ is taken to be trivial.

The weight space $\Wscr$ is the rigid space (over $E$, but it is well-defined over $\Qp$) parametrizing locally $\Qp$-analytic (equivalently, continuous) characters of $T_0 \simeq \left( \prod_{v | p} \Ocal_v^{\times} \right)^n$.
As $1 + \varpi_v \Ocal_v$ is isomorphic to $\left(\mu_{p^{\infty}} \cap F_v^{\times} \right) \times \Zp^{[F_v : \Qp]}$, $\Wscr$ is the product of an open polydisc of dimension $n[F : \Q]$ and a rigid space finite over $E$.

The construction in \cite{Loe} defines the $k$-analytic ($(G_k)_{k \geq 0}$ being a filtration of $G_0$) parabolic induction from $T_0$ to $G_0$ of the ``universal character'' $\chi : T_0 \rightarrow \Ocal(\Wscr)^{\times}$, denoted by $\Ccal(\mathscr{U},k)$ ($k$ big enough such that $\chi$ is $k$-analytic on the open affinoid $\mathscr{U}$), which interpolates $p$-adically the restriction to $\Galg'(\Qp)$ of algebraic representations of $\Galg' (\Qpbar)$.
From there one can define the spaces $M(e,\mathscr{U},k)$ (\cite[Definition 3.7.1]{Loe}) of $p$-adic automorphic forms (or overconvergent automorphic forms, by analogy with the rigid-geometric case of modular forms) above an open affinoid or a point $\mathscr{U}$ of $\Wscr$ which are $k$-analytic and fixed by the idempotent $e$.
This space has an action of $\Hcal^+$.
By \cite[Corollary 3.7.3]{Loe}, when considering $p$-adic automorphic forms which are eigenvectors for $[G_0uG_0]$ for some $u \in \Sigma^{++}$ and for a non-zero eigenvalue (``finite slope'' $p$-adic eigenforms), one can forget about $k$, and we will do so in the sequel.

\subsubsection{Existence and properties of the eigenvariety}

We choose the element
\[ \eta = \left( \begin{pmatrix}
\varpi_v^{-n} &  &  &  &  &  \\ 
 & \ddots &  &  &  &  \\ 
 &  & \varpi_v^{-1} &  &  &  \\ 
 &  &  & \varpi_v &  &  \\ 
 &  &  &  & \ddots &  \\ 
 &  &  &  &  & \varpi_v^n
\end{pmatrix} \right)_v \in \Sigma^{++} \]

\begin{theo}
\label{theo:existeigen}
There exists a reduced rigid space $\Xscr$ over $E$, together with an $E$-algebra morphism $\Psi : \Hcal^+ \rightarrow \Ocal(\Xscr)^{\times}$ and a morphism of rigid spaces $w : \Xscr \rightarrow \Wscr$ such that:
\begin{enumerate}
\item The morphism $\left(w,\Psi([G_0 \eta G_0])^{-1} \right): \Xscr \rightarrow \Wscr \times \mathbb{G}_m$ is finite
\item For each point $x$ of $\Xscr$, $\Psi \otimes w^{\natural} : \Hcal^+ \otimes_E \Ocal_{w(x)} \rightarrow \Ocal_x$ is surjective
\item For every finite extension $E'/E$, $\Xscr(E')$ is in bijection with the finite slope systems of eigenvalues of $\Hcal^+$ acting on the space of ``overconvergent'' automorphic forms, via evaluation of the image of $\Psi$ at a given point.
\end{enumerate}
Moreover, for any point $x \in |\Xscr|$, there is an arbitrarily small open affinoid $\mathscr{V}$ containing $x$ and an open affinoid $\mathscr{U}$ of $\Wscr$ such that $\mathscr{V} \subset w^{-1}(\mathscr{U})$, the morphism $w|_{\mathscr{V}} : \mathscr{V} \rightarrow \mathscr{U}$ is finite, and surjective when restricted to any irreducible component of $\mathscr{V}$.
\end{theo}
\begin{proof}
This is \cite[Theorems 3.11.2 and 3.12.3]{Loe}, except for the last assertion.
To prove it, we need to go back to the construction of the eigenvariety in \cite{Buzzard}.
Buzzard begins by constructing the Fredholm hypersurface $\mathscr{Z}$ (encoding only the value of $\Psi([G_0 \eta G_0])$), together with a flat morphism $\mathscr{Z} \rightarrow \Wscr$, before defining the finite morphism $\Xscr \rightarrow \mathscr{Z}$.
By \cite[Theorem 4.6]{Buzzard}, $\mathscr{Z}$ can be admissibly covered by its open affinoids $\mathscr{V}_0$ such that $w$ restricted to $\mathscr{V}_0$ induces a finite, surjective morphism to an open affinoid $\mathscr{U}$ of $\Wscr$, and $\mathscr{V}_0$ is a connected component of the pullback of $\mathscr{U}$.
We can assume that $\mathscr{U}$ is connected, and hence irreducible, since $\Wscr$ is normal.
The morphism $\mathscr{V}_0 \rightarrow \mathscr{U}$ is both open (since it is flat: \cite[Corollary 7.2]{Bosch}) and closed (since it is finite), so that any irreducible component of $\mathscr{V}_0$ is mapped onto $\mathscr{U}$.
This can be seen more naturally by observing that the irreducible components of $\mathscr{V}_0$ are also Fredholm hypersurfaces, by \cite[Theorem 4.3.2]{ConradIrr}.

By \cite[Proposition 6.4.2]{TheseG}, if $\mathscr{V}$ denotes the pullback to $\Xscr$ of $\mathscr{V}_0$, each irreducible component of $\mathscr{V}$ is mapped onto an irreducible component of $\mathscr{V}_0$ (more precisely, this is a consequence of \cite[Lemme 6.2.10]{TheseG}).
To conclude, we only need to show that if $x \in \mathscr{V}$, up to restricting $\mathscr{U}$, the connected component of $\mathscr{V}$ containing $x$ can be arbitrarily small.
This is a consequence of the following lemma.
\end{proof}

\begin{lemm}
Let $f : \Xscr_1 \rightarrow \Xscr_2$ be a finite morphism of rigid analytic spaces.
Then the connected components of $f^{-1}(U)$, for $U$ admissible open of $\Xscr_2$, form a basis for the canonical topology on $\Xscr_1$.
\end{lemm}
\begin{proof}
It is enough to consider the case $\Xscr_1 = \Sp A_1$, $\Xscr_2 = \Sp A_2$.
Let $x_1$ be a maximal ideal of $A_1$.
Then $f^{-1}\left( \{f(x_1) \} \right)= \left\{ x_1, \ldots , x_m \right\}$.
We choose generators $t_1, \ldots, t_n$ of $f(x_1)$, and $r_1^{(i)}, \ldots , r_{k_i}^{(i)}$ of $x_i$.
Using the maximum modulus principle, it is easily seen that $\Omega_{j,N} := \left \{ y \in \Xscr_2 \ |\ |t_j(y)| \geq p^{-N} \right\}_{j,N}$ is an admissible covering of the admissible open $\Xscr_2 \setminus \{ f(x)\}$ of $\Xscr_2$.
Let $V_M$ be the admissible open $\left\{ x \in \Xscr_1\ |\ \forall i, \exists k,\ |r_k^{(i)}(x)| \geq p^{-M} \right\}$, which is a finite union of open affinoids, hence quasi-compact.
Consequently, the admissible open sets
\begin{eqnarray*} U_{j,N} & := & V_M \cap f^{-1}\left(\Omega_{j,N}\right) \\
 & = & \left\{ x \in \Xscr_1\ |\ \forall i, \exists k,\ |r_k^{(i)}(x)| \geq p^{-M} \ \mathrm{and}\ |f^{\natural}(t_j)(x)| \geq p^{-N} \right\}_{j,N} \end{eqnarray*}
form an admissible covering of $V_M$.
Therefore there is an $N$ big enough so that
\[ V_M = \bigcup_{j=1}^r U_{j,N} \]
which implies that
\[ f^{-1} \left( \left\{ y \in \Xscr_2\ |\ |t_j(y)| \leq p^{-N-1} \right\} \right) \subset \bigcup_i \left\{ x \in \Xscr_1\ |\ \forall k,\ |r_k^{(i)}(x)| \leq p^{-M} \right\} \]
and when $M$ goes to infinity, the right hand side is the disjoint union of arbitrarily small affinoid neighbourhoods of the $x_i$.
\end{proof}

We define the algebraic points of $\Wscr(E)$ to be the ones of the form
\[ (x_{v,i})_{v,i} \mapsto \prod_{v,\sigma} \sigma \left( \prod_{i=1}^{n} x_{v,i}^{k_{v,\sigma,i}}  \right) \]
where $k_{v,\sigma,i}$ are integers, and such a point is called dominant if $k_{v,\sigma,1} \geq  k_{v,\sigma,2} \geq \ldots \geq k_{v,\sigma,n} \geq 0$.

Recall that a set $S \subset |\Xscr|$ is said to \emph{accumulate} at a point $x \in |\Xscr|$ if $x$ has a basis of affinoid neighbourhoods in which $S$ is Zariski dense.

\begin{prop}
\label{prop:zardensweights}
Let $(\phi_r)_r$ be a finite family of linear forms on $\R^A$ where $A$ is the set of triples $(v,\sigma,i)$ for $v$ a place of $F$ above $p$, $\sigma : F_v \rightarrow E$ and $1 \leq i \leq n$, and let $(c_r)_r$ be a family of elements in $\R_{\geq 0}$.
Assume that the open affine cone $C=\left\{ y \in \R^A\ |\ \forall r,\ \phi_r(y)>c_r \right\}$ is nonempty.
Then the set of algebraic characters in $C$ yields a Zariski dense set in the weight space $\Wscr$, which accumulates at all the algebraic points.
\end{prop}
\begin{proof}
\cite[Lemma 2.7]{Gunit}.
\end{proof}
In particular the property of being dominant or ``very regular'' can be expressed in this way.

By finiteness of $\Galg(F) \backslash \Galg(\A_{F,f}) \slash U$ for any open subgroup $U$ of $\Galg(\A_{F,f})$, if $\Pi$ is an automorphic representation of $\Galg(\A_F)$, the representation $\Pi_f$ is defined over $\iota_{\infty}(\bar{\Q})$.
Loeffler defines (\cite[Definition 3.9.1]{Loe}) the classical subspace of the space of $p$-adic automorphic forms above an algebraic and dominant point $w$ of the weight space.
This subspace is isomorphic to $\iota_p \iota_{\infty}^{-1} \left( e \left(\mathcal{C}^{\infty}(\Galg(F) \backslash \Galg(\A_F)) \otimes W^* \right)^{\Galg(F \otimes_{\Q} \R)} \right)$ as $\Hcal^{+}$-module, with $W$ the representation of $\Galg(F \otimes_{\Q} \R)$ which is the restriction of the algebraic representation of $\Galg' \times_{\Q} \C$ having highest weight $\iota_{\infty}^{-1} \iota_p (w)$.
The classical points of the eigenvariety are the ones having eigenvectors in the classical subspace.

We need to give an interpretation of classical points on the eigenvariety $\Xscr$, in terms of automorphic representations of $\Galg(\A_F)$.
Namely, there is a classical point $x \in \Xscr(E')$ defining a character $\Psi_x : \Hcal \rightarrow E'$ (here $E \subset E' \subset \Qpbar$)
if and only if there is an automorphic representation $\Pi = \otimes_v^{\prime} \Pi_v = \Pi_{\infty} \otimes \Pi_p \otimes \Pi_f^{(p)}$ of $\Galg(\A_F)$ such that:
\begin{itemize}
\item $\iota_p \iota_{\infty}^{-1} \left( \otimes_{v | \infty} \Pi_v \right)$ is the algebraic representation having highest weight $w(x)$;
\item $\iota_p\left((e^S \otimes e_S) \Pi_f^{(p)} \right)$ contains a non-zero vector on which $\Hcal^S \otimes \Hcal_S$ acts according to $\Psi_x$;
\item $\iota_p(e_{G_0} \Pi_p)$ contains a non-zero vector on which $\Hcal_p$ acts according to $\mu_{w(x)} \Psi_x$, where $\mu_{w(x)}([G_0 \xi G_0]) = w(x)(\xi)$ if $\xi \in \Sigma^+$.
\end{itemize}  
The twist by the character $\mu_{w(x)}$ is explained by the fact that the classical overconvergent automorphic forms are constructed by induction
of characters of the torus extended from $T_0$ (on which they are defined by $w$) to $T$ trivially on $\Sigma$.

\subsection{Unramified and ``completely refinable'' points}

\subsubsection{Small slope $p$-adic eigenforms are classical}

The algebraic and dominant points of $\Wscr$ are the ones of the form
\[ (x_{v,i})_{v,i} \mapsto \prod_{v,\sigma} \sigma \left( \prod_{i=1}^{n} x_{v,i}^{k_{v,\sigma,i}}  \right) \]
where $k_{v,\sigma,1} \geq k_{v,\sigma,2} \geq \ldots \geq k_{v,\sigma,n} \geq 0$ are integers.
The proof of the criterion given in \cite[Theorem 3.9.6]{Loe} contains a minor error, because it ``sees'' only the restriction of these characters to the maximal split torus $\Salg$ (over $\Qp$), and the BGG resolution has to be applied to \emph{split} semi-simple Lie algebras.

We correct it in the case of quasi-split reductive groups (in particular the restriction to a subfield of a quasi-split group remains quasi-split), and give a stronger criterion.
This criterion could be used on an eigenvariety for which only the weights corresponding to a given $p$-adic place of $F$ vary.
For this purpose we use the ``dual BGG resolution'' given in \cite{Jones}.
The proof will be very close to that of \cite[Propositions 2.6.3-2.6.4]{Loe}.
In the following $\Galg'$ could be any quasi-split reductive group over $\Qp$, and we could replace $E/\Qp$ by any extension splitting $\Galg'$.

Let $\Balg$ be a Borel subgroup of $\Galg'$, $\Salg$ a maximal split torus in $\Balg$, $\Talg$ the centralizer of $\Salg$, a maximal torus.
This determines an opposite Borel subgroup $\bar{\Balg}$ such that $\bar{\Balg} \cap \Balg = \Talg$.
Let $\Phi^+$ (resp. $\Delta$) be the set of positive (resp. simple) roots of $\Galg' \times_{\Qp} E$, with respect to the maximal torus $\Talg$ of the Borel subgroup $\Balg$.
One can split $\Delta = \sqcup_i \Delta_i$ where $\alpha,\beta$ belong to the same $\Delta_i$ if and only if $\alpha|_{\Salg} = \beta|_{\Salg}$ (equivalently, the $\Delta_i$ are the Galois orbits of $\Delta$).
Let $\Sigma$ be a subgroup of $\Talg(\Qp)$ supplementary to its maximal compact subgroup, and $\Sigma^+$ the submonoid consisting of the $z \in \Talg(\Qp)$ such that $|\alpha(z)| \geq 1$ for all $\alpha \in \Delta$.
For each $i$, define $\eta_i$ to be the element of $\Sigma^+ \slash \left( Z(\Galg')(\Qp) \cap \Sigma \right)$ generating $\cap_{j \neq i} \ker |\alpha_j(\cdot)|$ (here $\alpha_j$ denotes any element of $\Delta_j$, and $|\alpha_j(\cdot)|$ does not depend on this choice).

Assume that $G_0$ is a compact open subgroup of $\Galg'(\Qp)$ having an Iwahori factorization $\bar{N}_0 T_0 N_0$.
Using a lattice in the Lie algebra of $N$ and the exponential map, it is easily seen that $N_0$ admits a decreasing, exhaustive filtration by open subgroups $(N_k)_{k \geq 1}$ having a canonical rigid-analytic structure.
Moreover any ordering of $\Phi^+$ endows the Banach space of $\Qp$-analytic functions on $N_k$ taking values in $E$ with an orthonormal basis consisting of monomials on the weight spaces.

Let $\lambda$ be an algebraic and dominant weight of $\Talg \times_{\Qp} E$.
By \cite{Jones}, there is an exact sequence of $E[\mathbb{I}]$-modules, where $\mathbb{I}=G_0 \Sigma^+ G_0 = \bar{B}_0 \Sigma^+ N_0$ is the monoid generated by $G_0$ and $\Sigma^+$:
\begin{equation} \label{equ:BGG} 0 \rightarrow \Ind_{\bar{\Balg}}^{\Galg} (\lambda) \otimes \mbox{sm-Ind}_{\bar{B}_0}^{\bar{B}_0 N_0} 1  \rightarrow \mbox{la-Ind}_{\bar{B}}^{\bar{B} N_0}(\lambda) \rightarrow \bigoplus_{\alpha \in \Delta} \mbox{la-Ind}_{\bar{B}}^{\bar{B} N_0}(s_{\alpha}(\lambda+\rho)-\rho) \end{equation}
where $2 \rho = \sum_{\alpha \in \Phi^+} \alpha$, ``sm'' stands for ``smooth'' and ``la'' for ``locally analytic''. 
The relation with Loeffler's $\Ind(V)_k$ is $\mbox{la-Ind}_{\bar{B}}^{\bar{B} N_0}(\lambda) \otimes \lambda_{\mathrm{sm}}^{-1} = \ilim_k \Ind (E_{\lambda})_k$, where $\lambda_{\mathrm{sm}}$ is the character on $T$ which is trivial on its maximal compact subgroup and agrees with $\lambda$ on $\Sigma$.
Naturally $\Ind_{\bar{\Balg}}^{\Galg} (\lambda) \otimes \mbox{sm-Ind}_{\bar{B}_0}^{\bar{B}_0 N_0} 1 \otimes \lambda_{\mathrm{sm}}^{-1} = \ilim_k \Ind(E_{\lambda})_k^{\mathrm{cl}}$.

To prove a classicity criterion, we need to bound the action of $\eta_i$ on the factors of the RHS of \eqref{equ:BGG} twisted by $\lambda_{\mathrm{sm}}^{-1}$.
Let $n_{\alpha}=\alpha^{\vee}(\lambda) \in \N$ for $\alpha \in \Delta$, then $s_{\alpha}(\lambda+\rho)-\lambda-\rho=-(1+n_{\alpha})\alpha$.
The Banach space of $k$-analytic functions on $N_0$ is the direct sum of the spaces of analytic functions on $xN_k$, $x \in N_0/N_k$, and each of these spaces has an orthonormal (with respect to the supremum norm) basis $(v_{j,x})_{j \in J}$ where $J = \N^{\Phi^+}$ (monomials on the weights spaces).
This basis depends on the choice of a representative $x$, but if we fix $i$ and $x_0 \in N_0$, we can choose $\eta_i^{-1} x_0 \eta_i$ as a representative of its class.
Then if $\phi = \sum_{j} a_j v_{j,\eta_i^{-1} x_0 \eta_i}$ (with $a_j \rightarrow 0$) is an element of $\mbox{la-Ind}_{\bar{B}}^{\bar{B} N_0}(s_{\alpha}(\lambda+\rho)-\rho) \otimes \lambda_{\mathrm{sm}}^{-1}$, and $\xi \in N_k$,
\begin{eqnarray*} (\eta_i \cdot \phi)(x_0 \xi) & = & \eta_i^{-(1+n_{\alpha}) \alpha} \sum_{j \in J} a_j v_{j,\eta_i^{-1} x_0 \eta_i} (\eta_i^{-1} x_0 \xi \eta_i) \\
 & = & \sum_{j \in J} a_j \eta_i^{-(1+n_{\alpha}) \alpha - s(j)} v_{j,x_0} (x_0 \xi)
\end{eqnarray*}
where $s(j) = \sum_{\beta \in \Phi^+} j(\beta) \beta$.
This shows that $|\eta_i \cdot \phi| \leq |\alpha(\eta_i)|^{-(1+n_{\alpha})} |\phi|$, and so the operator $\eta_i$ has norm less than or equal to $|\alpha(\eta_i)|^{-(1+n_{\alpha})}$ on $\mbox{la-Ind}_{\bar{B}}^{\bar{B} N_0}(s_{\alpha}(\lambda+\rho)-\rho) \otimes \lambda_{\mathrm{sm}}^{-1}$.

We can then apply the exact functor which to an $E[\mathbb{I}]$-module $W$ associates the automorphic forms taking values in $W$, and take the invariants under the idempotent $e$ (this functor is left exact).
We obtain that $M(e,E_{\lambda})/M(e,E_{\lambda})_{\mathrm{cl}}$ (the space of $p$-adic automorphic forms modulo the classical automorphic forms) embeds in $\bigoplus_{\alpha \in \Delta} M_{\alpha}$ where each $M_{\alpha}$ is a Banach space on which the operator $[G_0 \eta_i G_0]$ has norm $\leq |\alpha(\eta_i)|^{-(1+n_{\alpha})}$.
The following criterion follows:
\begin{lemm}
\label{lemm:classcrit}
If an overconvergent eigenform $f \in M(e,E_{\lambda})$ satisfies $\left[G_0 \eta_i G_0\right] f = \mu_i f$ with $\mu_i \neq 0$ and
\[ v_p(\mu_i) < \inf_{\alpha \in \Delta_i} -(1+n_{\alpha}) v_p(\alpha (\eta_i)) \]
for all $i$, then $f$ is classical.
\end{lemm}

In the case of the symplectic group $\Galg'$, the family $(\eta_i)_i$ can be indexed by the couples $(v,i)$ where $v$ is a place of $F$ above $p$ and $1 \leq i \leq n$, and $\Delta_{v,i}$ is indexed by the embeddings $F_v \hookrightarrow E$.
Specifically, $\eta_{v,i}$ is trivial at all the places except for $v$, where it equals
\[ \mathrm{Diag}(x_1, \ldots, x_n, x_n^{-1}, \ldots, x_1^{-1}) \]
with $ x_j = \begin{cases} \varpi_v^{-1} & \mbox{ if } j \leq i \\ 1 & \mbox{ if } j>i \end{cases} $.

The conditions in the previous lemma can be written
\[ \begin{cases} v_p(\mu_{v,i}) < \frac{1}{e_v} \inf_{\sigma} (1+k_{v,\sigma,i}-k_{v,\sigma,i+1}) & \mbox{ for } i<n \\ v_p(\mu_{v,n}) < \frac{1}{e_v} \inf_{\sigma} \left( 2+2k_{v,\sigma,n} \right) . \end{cases} \]

\subsubsection{Representations having Iwahori-invariants and unramified principal series}

We recall results of Casselman showing that irreducible representations having Iwahori-invariants appear in unramified principal series, and giving the Atkin-Lehner eigenvalues in terms of the unramified character being induced.

In this subsection, we fix a place $v$ of $F$ above $p$.
Recall $I_v$ has an Iwahori decomposition $I_v=N_{v,0} T_{v,0} \bar{N}_{v,0}$.
As in \cite{Casselman}, if $(\Pi,V)$ is a smooth representation of $\Galg(F_v)$, $V(\bar{N}_v)$ is the subspace of $V$ spanned by the $\Pi(\bar{n})(x)-x$, $\bar{n} \in \bar{N}_v$, $V_{\bar{N}_v}=V/V(\bar{N}_v)$ and if $\bar{N}_{v,i}$ is a compact subgroup of $\bar{N}_v$, $V(\bar{N}_{v,i}) = \left\{ v \in V \ |\ \int_{\bar{N}_{v,i}} \Pi(\bar{n})(v) d \bar{n} = 0 \right\}$.

\begin{lemm}
\label{lemm:casselman}
Let $(\Pi,V)$ be an admissible representation of $\Galg(F_v)$ over $\C$.
Then the natural (vector space) morphism from $V^{I_v}$ to $\left( V_{\bar{N}_v} \right)^{T_{v,0}}$ is an isomorphism, inducing a $\Sigma_v^{+}$-equivariant isomorphism
\[ \Pi^{I_v} \xrightarrow{\sim} \left( \Pi_{\bar{N}_v} \right)^{T_{v,0}} \otimes \delta_{\bar{B}_v}^{-1} \]
where $\delta_{\bar{B}_v}$ denotes the modulus morphism of $\bar{B}_v$, and $u \in \Sigma_v^+$ acts on $\Pi^{I_v}$ by $[I_vuI_v]$.
\end{lemm}
\begin{proof}
Let $\bar{N}_{v,1}$ be a compact subgroup of $\bar{N}_v$ such that $V^{I_v} \cap V(\bar{N}_v) \subset V(\bar{N}_{v,1})$.
There is a $u \in \Sigma_v^{+}$ such that $u \bar{N}_{v,1} u^{-1} \subset \bar{N}_{v,0}$.
By \cite[Prop. 4.1.4]{Casselman}, and using the fact that $\left[ I_v u I_v \right]$ is invertible in the Hecke-Iwahori algebra, the natural morphism from $V^{I_v}$ to $V_{\bar{N}}^{T_{v,0}}$ is an isomorphism (of vector spaces).

Lemmas 4.1.1 and 1.5.1 in \cite{Casselman} allow to compute the action of $\Sigma_v^{+}$.
\end{proof}

\begin{coro}
Any smooth irreducible representation of $\Galg(F_v)$ over $\C$ having Iwahori invariants is a subquotient of the parabolic induction (from $\bar{B}_v$) of a character of the torus $T_v$,
which is unique up to the action of $W(T_v,\Galg(F_v))$, and unramified.
\end{coro}
\begin{proof}
$\Pi$ is a subquotient of the parabolic induction of a character of the torus $T_v$ if and only if $\Pi_{\bar{N}_v} \neq 0$, which is true by the previous lemma.
The geometrical lemma \cite[2.12]{BernZ} shows that if $\chi$ is a smooth character of $T_v$,
\[ \left( \Ind_{\bar{B}_v}^{\Galg(F_v)} \chi \right)_{\bar{N}_v}^{\mathrm{ss}} \simeq \bigoplus_{w \in W(T_v,\Galg(F_v))} \chi^{w} \delta^{1/2}_{\bar{B}_v} \]
\end{proof}
Since $*_{\bar{N}}$ is left adjoint to \emph{non-normalized} induction, the first argument in the proof shows that $\Pi$ is actually
a subrepresentation of $\Ind_{\bar{B}_v}^{\Galg(F_v)}$ for at least one $\chi$ in the orbit under $W(T_v,\Galg(F_v))$.
In that case we will say that $(\Pi,\chi)$ is a \emph{refinement} of $\Pi$.
Note that up to the action of $W(T_v,\Galg(F_v))$, there is a unique $\chi$ such that $\Pi$ is a subquotient of $\Ind_{\bar{B}_v}^{\Galg(F_v)}$.

\subsubsection{Most points of the eigenvariety arise from unramified, completely refinable representations}

We will need a result of Tadi\'c, characterizing the \emph{irreducible} principal series.
If $\chi_1, \ldots, \chi_n$ are characters of $F_v^{\times}$, we denote simply by $\chi = (\chi_1, \ldots, \chi_n)$ the character of $T_v$ which maps
\[ \begin{pmatrix}
x_1 &  &  &  &  &  \\ 
 & \ddots &  &  &  &  \\ 
 &  & x_n &  &  &  \\ 
 &  &  & x_n^{-1} &  &  \\ 
 &  &  &  & \ddots &  \\ 
 &  &  &  &  & x_1^{-1}
\end{pmatrix} \]
to $\prod_{i=1}^n \chi_i(x_i)$.
Let $\nu$ be the unramified character of $F_v^{\times}$ such that $\nu(\varpi_v)=\left| \F_v \right|^{-1}$.

\begin{theo}
\label{theo:irrind}
Let $\chi = (\chi_1, \ldots, \chi_n)$ be a character of $T_v$.
Then $\Ind_{\bar{B}_v}^{\Sp_{2n}(F_v)} \chi$ is irreducible if and only if the following conditions are satisfied
\begin{enumerate}
 \item For all $i$, $\chi_i$ is not of order $2$.
 \item For all $i$, $\chi_i \neq \nu^{\pm 1}$.
 \item For all distinct $i,j$, $\chi_i \chi_j^{-1} \neq \nu^{\pm 1}$ and $\chi_i \chi_j \neq \nu^{\pm 1}$.
\end{enumerate}
\end{theo}
\begin{proof}
\cite[Theorem 7.1]{Tadic}
\end{proof}

\begin{defi}
An irreducible representation $\Pi_v$ of $\Galg(F_v)$ is \emph{completely refinable} if it is isomorphic to $\Ind_{\bar{B}_v}^{\Sp_{2n}(F_v)} \chi$ for some unramified character $\chi$.

An automorphic representation $\Pi$ of $\Galg(\A_F)$ is completely refinable if $\Pi_v$ is completely refinable for any $v | p$.
\end{defi}

Note that completely refinable representations are unramified (for any choice of hyperspecial subgroup).
A representation $\Pi_v$ is completely refinable if and only if $(\Pi_v)_{\bar{N}_v}^{\mathrm{ss}}$ is the sum of $|W(T_v,\Galg(F_v))|$ unramified characters.

Recall that classical points on the eigenvariety are determined by an automorphic representation $\Pi$ together with a refinement of each $\Pi_v$, $v|p$.
Completely refinable automorphic representations are the ones giving the greatest number of points on the eigenvariety.
When one can associate Galois representations to automorphic representations, each refinement of $\Pi$ comes with a ``$p$-adic family'' of Galois representations going through the same one.

\begin{prop}
\label{prop:zardens}
Let $f_1, \ldots, f_r \in \Ocal(\Xscr)^{\times}$.
The set $\mathcal{S}$ of points corresponding to completely refinable, unramified classical points at which
\begin{equation} \min_{v,\sigma} \min \{ k_{v,\sigma,1}-k_{v,\sigma,2}, \ldots, k_{v,\sigma,n-1}-k_{v,\sigma,n}, k_{v,\sigma,n} \} \geq \max \{v_p(f_1),\ldots ,v_p(f_n) \} \label{equ:zardenseineq} \end{equation}
is Zariski dense and accumulates at all the algebraic points.
\end{prop}
Compare \cite[Proposition 6.4.7]{TheseG}, \cite[Corollary 3.13.3]{Loe}.
\begin{proof}
The hypotheses in the classicality criterion \ref{lemm:classcrit} and the ones in Theorem \ref{theo:irrind} are implied by inequalities of the form \ref{equ:zardenseineq}.
First we prove the accumulation property.
We can restrict to open affinoids $\mathscr{V}$ of the eigenvariety, and hence assume that the right hand side of \ref{equ:zardenseineq} is replaced by a constant.
By Theorem \ref{theo:existeigen}, $\mathscr{V}$ can be an arbitrarily small open affinoid containing an algebraic point $x$ of $\Xscr$, such that there is open affinoid $\mathscr{U}$ of $\Wscr$ such that $\mathscr{V} \subset w^{-1}(\mathscr{U})$, the morphism $w|_{\mathscr{V}} : \mathscr{V} \rightarrow \mathscr{U}$ is finite, and surjective when restricted to any irreducible component of $\mathscr{V}$.
By Proposition \ref{prop:zardensweights}, the algebraic weights satisfying \ref{equ:zardenseineq} are Zariski dense in the weight space $\Wscr$ and accumulate at all the algebraic points of $\Wscr$.
\cite[Lemme 6.2.8]{TheseG} shows that $\mathcal{S} \cap \mathscr{V}$ is Zariski-dense in $\mathscr{V}$.

Each irreducible component $\Xscr'$ of $\Xscr$ is mapped onto a Zariski-open subset of a connected component of $\Wscr$, by \cite[Corollaire 6.4.4]{TheseG} (which is a consequence of the decomposition of a Fredholm series into a product of prime Fredholm series, \cite[Corollary 4.2.3]{ConradIrr}), so $\Xscr'$ contains at least one algebraic point (the algebraic weights intersect all the connected components of $\Wscr$), and hence the Zariski closure of $\mathcal{S} \cap \Xscr'$ contains an open affinoid of $\Xscr'$, which is Zariski dense in $\Xscr'$.
\end{proof}

\section{Galois representations associated with automorphic representations of symplectic groups}

\subsection{A consequence of Arthur's description of the discrete spectrum for classical groups}

\subsubsection{Automorphic self-dual representations of $\GL_{2n+1}$ of orthogonal type}

According to Arthur's conjectural parametrization of discrete automorphic representations, each such representation of $\Galg(\A_F)$ should be part of an A-packet corresponding to a discrete parameter, which is a representation
\[ \mathcal{L}_F \times \mathrm{SL}_2(\C) \rightarrow \mathrm{SO}_{2n+1}(\C) \]
such that (among other conditions) the commutant of the image is finite.

The standard embedding $\mathrm{SO}_{2n+1}(\C) \hookrightarrow \GL_{2n+1}(\C)$ ``transfers'' this parameter to a parameter of $\GL_{2n+1}/F$, which is not discrete in general, and thus it corresponds to an automorphic representation of $\GL_{2n+1}(\A_F)$.
Here we define an automorphic representation $\pi$ of $\GL_{N}(\A_F)$ as a formal sum of discrete automorphic representations $\pi_i$ of $\GL_{n_i}$ such that $\sum_i n_i = N$.
We will write $\pi = \boxplus_i \pi_i$.
By \cite{MoeWal}, each $\pi_i$ is the Langlands quotient of the parabolic induction of twists of a single cuspidal representation by powers of $|\det|$.
We will not need this generality, as we will force the representations $\pi_i$ to be cuspidal in the sequel.

Since $\pi$ comes from a self-dual parameter, it is self-dual: $\pi^{\vee} \simeq \pi$.
Even though $\pi$ is not discrete in general, the discreteness of the parameter which takes values in $\mathrm{SO}_{2n+1}$ implies that the $\pi_i$'s are self-dual.

If $\Pi = \otimes_v \Pi_v$ is an automorphic representation of $\Galg(\A_F)$, then for any archimedean place $v$ of $F$, the local Langlands parmeter of $\Pi_v$ (composed with $\SO_{2n+1}(\C) \hookrightarrow \GL_{2n+1}(\C)$) is of the form:
\[ \mathcal{LL}(\Pi_v) \simeq \epsilon^n \oplus \bigoplus_{i=1}^n \Ind_{W_{\C}}^{W_{\R}} \left( z \mapsto (z/\bar{z})^{r_i} \right) \]
where $\epsilon$ is the only non-trivial character of $W_{\C}/W_{\R}$, and the $r_i$ are integers, with $r_n > r_{n-1} > \ldots > r_1 > 0$.
We define $A_{\Sp_{2n}}$ to be the set of automorphic representations such that for each infinite place $v$ of $F$, $r_1 \geq 2$ and $r_{i+1} \geq r_i + 2$.
The equivalence above is meant as representations of $W_{\R}$, although $\mathcal{LL}(\Pi_v)$ is a parameter taking values in $\SO_{2n+1}(\C)$ (the two notions coincide).

Similarly, let $A_{\GL_{2n+1}}$ be the set of formal sums of self-dual cuspidal representations $\pi = \boxplus_i \pi_i = \otimes_v \pi_v$ of $\GL_{2n+1}(\A_F)$ such that for each infinite place $v$ of $F$,
\[ \mathcal{LL}(\pi_v) \simeq \epsilon^n \oplus \bigoplus_{i=1}^n \Ind_{W_{\C}}^{W_{\R}} \left( z \mapsto (z/\bar{z})^{r_i} \right) \]
where the $r_i$'s are integers,such that $r_1 \geq 2$, $r_{i+1} \geq r_i + 2$, and such that the product of the central characters of the $\pi_i$'s is trivial.

These inequalities are imposed to ensure that the corresponding global parameters are trivial on Arthur's $\SL_2(\C)$, to simplify the statements.
That is why we take formal sums of \emph{cuspidal} (not discrete) representations.

Note that there is no non-zero alternate bilinear form preserved by such a parameter (one could say that the parameter is ``completely orthogonal'').

\begin{assu}
\label{prop:transfersp}
For any $\Pi \in A_{\Sp_{2n}}$, there is a $\pi \in A_{\GL_{2n+1}}$, such that the local Langlands parameters match at the infinite places, and for any finite place $v$ of $F$,
$\pi_v$ is unramified if $\Pi_v$ is unramified, and in that case the local parameters match, by means of the inclusion $\SO_{2n+1}(\C) \subset \GL_{2n+1}(\C)$.
\end{assu}

\subsubsection{$p$-adic Galois representations associated with RLASDC representations of $\GL_N$}

An automorphic cuspidal representation $\pi$ of $\GL_N(\A_F)$ is said to be \emph{L-algebraic} if for any infinite place $v$ of $F$, the restriction of $\mathcal{LL}(\pi_{v})$ to $\C^{\times}$ is of the form
\[ z \mapsto \mathrm{Diag}\left( \left(z^{a_{v,i}}\bar{z}^{b_{v,i}}\right)_i \right) \]
where $a_i,b_i \in \Z$.
By the ``purity lemma'' \cite[Lemme 4.9]{Clozel}, $a_{v,i}+b_{v,i}$ does not depend on $v,i$.
We will say that $\pi$ is L-algebraic \emph{regular} if for any $v$ as above, the $a_{v,i}$ are distinct.
By purity, this implies that if $v$ is real,
\begin{align*}
\mathcal{LL}(\pi_v)|\cdot|^{-s} = \begin{cases}
\epsilon^{e} \oplus_i \Ind^{W_{\R}}_{W_{\C}} \left( z \mapsto (z/\bar{z})^{a'_{v,i}} \right) & \text{ if $N$ is odd, with $e=0,1$} \\
\oplus_i \Ind^{W_{\R}}_{W_{\C}} \left( z \mapsto (z/\bar{z})^{a'_{v,i}} \right) & \text{ if $N$ is even}
\end{cases}
\end{align*}
for some integer $s$, and integers $0<a'_{v,1}<\ldots<a'_{v,\lfloor N/2 \rfloor}$.

As a special case of \cite[Theorem 4.2]{CheHar} (which builds on previous work of Clozel, Harris, Kottwitz, Labesse, Shin, Taylor), we have the following theorem
\begin{theo}
\label{theo:galrepgl}
Let $\pi$ be a regular L-algebraic, self-dual, cuspidal (RLASDC) representation of $\GL_{2n+1}(\A_F)$.
Then $\pi$ is L-arithmetic, and there is a continuous Galois representation
\[ \rho_{\iota_p,\iota_{\infty}}(\pi) : G_F \longrightarrow \GL_{2n+1}(\Qpbar) \]
such that if $v$ is a finite place of $F$ and $\pi_v$ is unramified,
\begin{enumerate}
\item if $v$ is coprime to $p$, then $\rho_{\iota_p,\iota_{\infty}}(\pi)|_{G_{F_v}}$ is unramified, and
\[ \det \left( T \mathrm{Id} - \rho_{\iota_p,\iota_{\infty}}(\pi)(\mathrm{Frob}_v) \right) = \iota_p \iota_{\infty}^{-1} \det \left( T \mathrm{Id} - A \right) \]
where $A \in \GL_N(\C)$ is associated with $\pi_v$ via the Satake isomorphism.
\item if $v$ lies above $p$, $\rho_{\iota_p,\iota_{\infty}}(\pi)|_{G_{F_v}}$ is crystalline.
The associated filtered $\varphi$-module (over $F_{v,0} \otimes_{\Qp} \Qpbar$) is such that
\[\det {}_{\Qpbar} \left( T \mathrm{Id} - \varphi^{f_v} \right) = \iota_p \iota_{\infty}^{-1} \det \left( T \mathrm{Id} - A \right)^{f_v}\]
where $A \in \GL_N(\C)$ is associated with $\pi_v$ via the Satake isomorphism.
For any $\sigma : F_v \rightarrow \Qpbar$, the $\sigma$-Hodge-Tate weights are the $a_{w,i}$, where $w$ is the real place of $F$ defined by $\sigma$, $\iota_p$ and $\iota_{\infty}$.
\end{enumerate}
\end{theo}

The power $f_v$ appearing at places above $p$ may seem more natural to the reader (and will actually disappear) after reading subsubsection \ref{subsub:crys}.

Combining this theorem with the transfer detailed in the last section, we obtain
\begin{coro}
\label{coro:galrepsp}
Let $\Pi$ be an automorphic representation of $\Galg(\A_F)$, whose weights $k_{w,1} \geq k_{w,2} \geq \ldots k_{w,n} \geq 0$ at the real places $w$ are far from the walls ($\Pi \in A_{\Sp_{2n}}$ is enough), and unramified at the places above $p$.
There exists a continuous semisimple Galois representation
\[ \rho_{\iota_p,\iota_{\infty}}(\Pi) : G_F \longrightarrow \GL_{2n+1}(\Qpbar) \]
such that for any finite place $v$ of $F$ such that $\Pi_v$ is unramified
\begin{enumerate}
\item if $v$ is coprime to $p$, then $\rho_{\iota_p,\iota_{\infty}}(\Pi)|_{G_{F_v}}$ is unramified, and
\[ \det \left( T \mathrm{Id} - \rho_{\iota_p,\iota_{\infty}}(\Pi)(\mathrm{Frob}_v) \right) = \iota_p \iota_{\infty}^{-1} \det \left( T \mathrm{Id} - A \right) \]
where $A \in \GL_N(\C)$ is associated with $\Pi_v$ via the Satake isomorphism.
\item if $v$ lies above $p$, $\rho_{\iota_p,\iota_{\infty}}(\Pi)|_{G_{F_v}}$ is crystalline.
The associated filtered $\varphi$-module is such that
\[\det {}_{\Qpbar} \left( T \mathrm{Id} - \varphi^{f_v} \right) = \iota_p \iota_{\infty}^{-1} \det \left( T \mathrm{Id} - A \right)^{f_v}\]
where $A \in \SO_{2n+1}(\C) \subset \GL_{2n+1}(\C)$ is associated with $\Pi_v$ via the Satake isomorphism.
For any $\sigma : F_v \rightarrow \Qpbar$, the $\sigma$-Hodge-Tate weights are $k_{w,1}+n > k_{w,2}+n-1> \ldots > k_{w,1}+1 > 0 > - k_{w,1}-1 > \ldots > -k_{w,1}-n$, where $w$ is the real place of $F$ defined by $\sigma$, $\iota_p$ and $\iota_{\infty}$.
\end{enumerate}
\end{coro}
\begin{proof}
There is an automorphic representation $\pi = \boxplus_i \pi_i$ of $\GL_{2n+1}(\A_F)$ corresponding to $\Pi$ by Assumption \ref{prop:transfersp}, obtained by induction from distinct cuspidal representations $\pi_i$.
Let $\rho_{\iota_p,\iota_{\infty}}(\Pi) = \oplus_i \rho_{\iota_p,\iota_{\infty}}(\pi_i)$.
\end{proof}
Note that in that case, since $\Pi_{\infty}$ is C-algebraic, $\Pi$ is obviously C-arithmetic (which is equivalent to L-arithmetic in the case of $\Sp_{2n}$), and thus the coefficients of the polynomials appearing in the corollary lie in a finite extension of $\Q$.

\subsubsection{The Galois pseudocharacter on the eigenvariety}

To study families of representations, it is convenient to use \emph{pseudorepresentations} (or \emph{pseudocharacters}), which are simply the traces of semi-simple representations when the coefficient ring is an algebraically closed field of characteristic zero.
We refer to \cite{Tay2} for the definition, and \cite[Theorem 1]{Tay2} is the ``converse theorem'' we will need.

On $\Ocal(\Xscr)$, we put the topology of uniform convergence on open affinoids.

The Zariski-density of the classical points at which we can define an attached Galois representation implies the following
\begin{prop}
There is a continuous pseudocharacter $T : G_F \rightarrow \Ocal(\Xscr)$, such that at every classical unramified point of the eigenvariety having weight far from the walls, $T$ specializes to the character of the Galois representation associated with the automorphic representation by Corollary \ref{coro:galrepsp}.
\end{prop}
\begin{proof}
This is identical to the unitary case, and thus is a consequence of \cite[Proposition 7.1.1]{TheseG}, by Proposition \ref{prop:zardens}.
\end{proof}

Thus at any (classical or not) point of the eigenvariety, there is an attached Galois representation.

\subsection{Galois representations stemming from symplectic forms are generically almost irreducible}

\subsubsection{Crystalline representations over $\Qpbar$}
\label{subsub:crys}

We fix a finite extension $K$ of $\Qp$, and denote $K_0$ the maximal unramified subextension, $e = \left[K : K_0 \right]$, $f = \left[K_0 : \Qp \right]$.
Let $\rho : G_K \rightarrow \GL(V)$ be a continuous representation of the absolute Galois group of $K$, where $V$ is a finite dimensional vector space over $L$, a finite Galois extension of $\Qp$.
We will take $L$ to be big enough so as to be able to assume in many situations that $L=\Qpbar$.
For example, we can assume that $L$ is an extension of $K$, and that $\rho$ has a composition series $0 = V_1 \subset \ldots \subset V_r = V$ such that each quotient $V_{i+1}/V_i$ is absolutely irreducible.

For any such $\rho$, we denote $\Dcris(V)=\left( \Bcris \otimes_{\Qp} V \right)^{G_K}$.
From now on we assume that $\rho$ is a crystalline representation, which means that $\dim_{K_0} \Dcris(V) = \dim_{\Qp} V$.
It is well-known that $\Dcris(V)$ is a filtered $\varphi$-module over $K$, and since $V$ is a vector space over $L$, $\Dcris(V)$ is a $\varphi$-module over $K_0 \otimes_{\Qp} L$, and $\Ddr(V)=K \otimes_{K_0} \Dcris(V)$ is a module over $K \otimes_{\Qp} L$ with a filtration by projective submodules.

We have a natural decomposition $K_0 \otimes_{\Qp} L \simeq \prod_{\sigma_0 \in \Upsilon_0} L_{\sigma_0}$ with $\Upsilon_0 = \Hom_{\Qp\mathrm{-alg.}}(K_0,L)$ and $L_{\sigma_0} \simeq L$, given by the morphisms $\sigma_0 \otimes \mathrm{Id}_L$.
Similarly, $K \otimes_{\Qp} L \simeq \prod_{\sigma \in \Upsilon} L_{\sigma}$ with $\Upsilon = \Hom_{\Qp\mathrm{-alg.}}(K,L)$.

Hence we have decompositions
\[ \Dcris(V) = \prod_{\sigma_0 \in \Upsilon_0} \Dcris(V)_{\sigma_0},\ \ \Ddr(V) = \prod_{\sigma \in \Upsilon} \Ddr(V)_{\sigma} .\]
The operator $\varphi$ restricts as linear isomorphisms from $\Dcris(V)_{\sigma_0}$ to $\Dcris(V)_{\sigma_0 \circ \varphi^{-1}}$, and so $\varphi^f$ is a $L_{\sigma_0}$-linear automorphism on each $\Dcris(V)_{\sigma_0}$, which are isomorphic as vector spaces over $L$ equipped with the linear automorphism $\varphi^f$.

Each $\Ddr(V)_{\sigma}$ comes with a filtration, and hence defines $\dim_L V = N$ Hodge-Tate weights $k_{\sigma,1} \leq \ldots \leq k_{\sigma,N}$ (the jumps of the filtration).

Although we will not use it, it should be noted that by \cite[Proposition 3.1.1.5]{BreMez}, to verify the weak admissibility of a filtered $\varphi$-module $D$ over $K$ with an action of $L$ commuting with $\varphi$ and leaving the filtration stable, it is enough to check the inequality $t_N(D') \geq t_H(D')$ for sub-$K_0 \otimes L$-modules stable under $\varphi$.

If $\varphi^f$ has eigenvalues $\varphi_1, \ldots , \varphi_N$, with $v_p(\varphi_1) \leq \ldots \leq v_p(\varphi_n)$, we can in particular choose $D' = \oplus_{i \leq j} \ker (\varphi^f - \varphi_i)$ (if the eigenvalues are distinct, but even if they are not, we can choose $D'$ such that $\varphi^f|_{D'}$ has eigenvalues $\varphi_1, \ldots, \varphi_j$, counted with multiplicities).
The worst case for the filtration yields the inequalities
\begin{align*}
& v_p(\varphi_1) \geq \frac{1}{e} \sum_{\sigma} k_{\sigma,1} \\
& v_p(\varphi_1 \varphi_2) \geq \frac{1}{e} \sum_{\sigma} k_{\sigma,1} + k_{\sigma,2} \\
& \vdots 
\end{align*}
In the sequel, we will only use these inequalities, and we will not be concerned with the subtleties of the filtrations.

\subsubsection{Variation of the crystalline Frobenius on the eigenvariety}

In this section we explicit the formulas relating the eigenvalues of the crystalline Frobenius at classical, unramified points of the eigenvariety and the eigenvalues of the Hecke-Iwahori operators acting on $p$-adic automorphic forms.
Let $x$ be a classical point on the eigenvariety.
There is an automorphic representation $\Pi$ of $\Galg(\A_F)$ such that $\iota_p \iota_{\infty}^{-1} (\Pi_{\infty})$ is the representation having highest weight $w(x)$.
Assume that $\Pi_p$ is unramified.
The point $x$ defines a refinement of $\Pi_p$, that is an unramified character $\chi_x : T_0 \rightarrow \C^{\times}$ such that $\Pi_p \hookrightarrow \Ind_{\bar{B}}^{\Galg'(\Qp)} \chi_x$, or equivalently the character $ \delta_{\bar{B}}^{1/2} \chi_x$ appearing in $(\Pi_p)_{\bar{N}}$.
By \ref{lemm:casselman}, for any $u \in \Sigma^+$, $\mu_{w(x)} \Psi_x|_{\Hcal_p} = (\iota_p \circ \iota_{\infty}^{-1} \circ \chi_x) \delta_{B}^{1/2}$.

The diagonal torus in $\SO_{2n+1}(\C)$ and the identification of it with the dual of the diagonal torus of $\Sp_{2n}/F_v$ being fixed, the character $\chi_x$ is mapped by the unramified Langlands correspondence for tori to $y=(y_v)_{v|p}$ with $y_v=\mathrm{Diag}(y_{1,v},\ldots,y_{n,v},1,y_{n,v}^{-1},\ldots,y_{1,v}^{-1})$, and $y_{v,i}=\chi_x(\mathrm{Diag}(1,\ldots,\varpi_v,\ldots,1,1,\ldots,\varpi_v^{-1},\ldots,1))$ ($\varpi_v$ being the $i$-th element).
Thus the linearization of the crystalline Frobenius $\varphi^{f_v}$ on $\Dcris(\rho_{\iota_p,\iota_{\infty}}(\pi)|_{G_{F_v}})_{\sigma_0}$ (for any choice of $\sigma_0 : F_v \rightarrow E$ in $\Upsilon_{0,v}$) has eigenvalues
\[ \iota_p \iota_{\infty}^{-1} (y_{v,i}) = q_v^{n+1-i} \phi_{v,n+1-i}(x) \prod_{\sigma \in \Upsilon_v} \sigma(\varpi_v)^{k_{v,\sigma,i}}  \]
and their inverses, together with the eigenvalue $1$.
Here $\phi_{v,n+1-i} \in \Ocal(\Xscr)$ is defined by
\[ \phi_{v,n+1-i} = \frac{\Psi \left([G_0 u_{i-1} G_0] \right)}{\Psi \left([G_0 u_i G_0] \right)} \]
with $u_i = \mathrm{Diag(\varpi_v^{-1},\ldots,\varpi_v^{-1},1,\ldots,1,\varpi_v,\ldots,\varpi_v)}$ (the last $\varpi_v^{-1}$ is the $i$-th element), and $k_{v,\sigma,i}$ the integers defining the weight $w(x)$.

Assume furthermore that $\Pi_p$ admits another refinement $\chi_{x'} = \chi_x^{a}$ for some $a=(a_v)_{v|p}$ in the Weyl group $W(\Galg'(\Qp),\Talg(\Qp))=\prod_v W(\Galg(F_v),T_v)$.
Each $W(\Galg(F_v),T_v)$ can be identified with the group of permutations $a_v : \left\{ -n, \ldots, n \right\} \rightarrow \left\{ -n, \ldots, n \right\}$ such that $a_v(-i)=-a_v(i)$ for all $i$, acting by
\[ a_v(\mathrm{Diag}(x_1,\ldots,x_n,x_n^{-1},\ldots,x_1^{-1})) = \mathrm{Diag}(x_{a_v^{-1}(1)},\ldots,x_{a_v^{-1}(n)},x_{a_v^{-1}(-n)},\ldots,x_{a_v^{-1}(1)}) \]
on $T_v$, where for commodity we set $x_{-i}=x_i^{-1}$ for $i<0$.
Similarly we define $k_{v,\sigma,-i}=-k_{v,\sigma,i}$ and $\phi_{v,-i}=\phi_{v,i}^{-1}$.
We also set $k_{v,\sigma,0}=0$, $\phi_{v,0}=1$.
The equality $\chi_{x'} = \chi_x^{a}$ can also be written
\[ q_v^{(n+1)\mathrm{sign}(w(i))-w(i)} \phi_{v,n+1-w(i)}(x) \prod_{\sigma \in \Upsilon_v} \sigma(\varpi_v)^{k_{v,\sigma,w(i)}} = \\
q_v^{n+1-i} \phi_{v,n+1-i}(x') \prod_{\sigma \in \Upsilon_v} \sigma(\varpi_v)^{k_{v,\sigma,i}} \]
which is valid for any $-n \leq i \leq n$ if we set $\mathrm{sign}(i)=-1$ (resp.\ $0$, $1$) if $i$ is negative (resp.\ zero, positive), and equivalent to
\[ \phi_{v,n+1-i}(x')= \phi_{v,n+1-w(i)}(x) q_v^{i-w(i)+(n+1)(\mathrm{sign}(i)-\mathrm{sign}(w(i)))} \prod_{\sigma \in \Upsilon_v} \sigma(\varpi_v)^{k_{v,\sigma,w(i)}-k_{v,\sigma,i}} .\]
This last formula will be useful in the proof of the main result.

\subsubsection{Main result}

\begin{lemm}
Let $K$ be a finite extension of $\Qp$, and let $\rho : G_K \rightarrow \GL_{N}(\Qpbar)$ be a crystalline representation.
Let $(D,\varphi,\Fil^i D \otimes_{K_0} K )$ be the associated filtered $\varphi$-module.
Let $\kappa_{\sigma,1} \leq \ldots \leq \kappa_{\sigma,N}$ be the Hodge-Tate weights associated with the embedding $\sigma : K \hookrightarrow \Qpbar$.
Let $\varphi_1, \ldots , \varphi_N$ be the eigenvalues of the linear operator $\varphi^{f}$ (on any of the $D_{\sigma_0}$, $\sigma_0 \in \Upsilon_0$), and suppose they are distinct.
Finally, assume that for some $\tau \in \Upsilon$, for all $i$,
\[ \left| v_p(\varphi_i) - \frac{1}{e} \sum_{\sigma \in \Upsilon} \kappa_{\sigma,i} \right| \leq \frac{1}{eN} \min_{1 \leq j \leq N-1} \kappa_{\tau,j+1} - \kappa_{\tau,j}   .\]
Then if $D' \subset D$ is an admissible sub-$\varphi$-module over $K_0 \otimes_{\Qp} \Qpbar$ (corresponding to a subrepresentation), there is a subset $I$ of $\left\{ 1, \ldots , N \right\}$ such that $D'$ has $\varphi^f$-eigenvalues $\left( \varphi_i \right)_{i \in I}$ and $\tau$-Hodge-Tate weights $\left( \kappa_{\sigma,i} \right)_{i \in I}$.
\end{lemm}
\begin{proof}
Since the eigenvalues of $\varphi^f$ are distinct, and $D'$ is stable under $\varphi$, there is a subset $I$ of $\left\{ 1, \ldots , N \right\}$ such that $D' = \ker \prod_{i \in I} \left( \varphi^f - \varphi_i \right)$.
There are unique increasing functions $\theta_{1,\sigma} : I \rightarrow \left\{ 1, \ldots , N \right\}$ such that the $\sigma$-weights of $D'$ are the $\kappa_{\sigma,\theta_{1,\sigma}(i)}$, for $i \in I$.
By ordering similarly the weights of $D/D'$, we define increasing functions $\theta_{2,\sigma} : \left\{ 1, \ldots, N \right\} \setminus I \rightarrow \left\{ 1, \ldots, N \right\}$, and we can glue the $\theta_{\cdot,\sigma}$ to get bijective maps $\theta_{\sigma} : \left\{ 1, \ldots , N \right\} \rightarrow \left\{ 1, \ldots, N \right\}$.
We will show that $\theta_{\tau} = \mathrm{Id}$

We now write the admissibility condition for $D'$ and $D/D'$.
Let $i_1$ be the smallest element of $I$.
Then $\ker \left( \varphi^f - \varphi_{i_1} \right)$ is a sub-$\varphi$-module of $D'$.
Its induced $\sigma$-weight is one of the $\kappa_{\sigma,\theta_{\sigma}(i)}$ for $i \in I$, thus it is greater than or equal to $\kappa_{\sigma,\theta_{\sigma}(i_1)}$.
This implies that $v_p(\varphi_{i_1}) \geq 1/e \sum_{\sigma \in \Upsilon} \kappa_{\sigma,\theta_{\sigma}(i_1)}$.
We can proceed similarly for the submodules
\[\ker \left( \left(\varphi^f - \varphi_{i_1} \right) \ldots \left(\varphi^f - \varphi_{i_r} \right) \right)\]
(where the $i_{\cdot}$ are the ordered elements of $I$), to get the inequality
\[ \sum_{1 \leq x \leq r} v_p(\varphi_{i_x}) \geq \frac{1}{e} \sum_{1 \leq x \leq r}  \sum_{\sigma \in \Upsilon} \kappa_{\sigma,\theta_{\sigma}(i_x)} \]
The same applies to $D/D'$, and by adding both inequalities, we finally get
\[ \sum_{1 \leq i \leq s} v_p(\varphi_{i}) \geq \frac{1}{e} \sum_{1 \leq i \leq s} \sum_{\sigma \in \Upsilon} \kappa_{\sigma,\theta_{\sigma}(i)} \]
We now isolate $\tau$, using the fact that $\sum_{1 \leq i \leq s} \kappa_{\sigma,\theta_{\sigma}(i)} \geq \sum_{1 \leq i \leq s} \kappa_{\sigma,i}$ for $\sigma \neq \tau$, and obtain the inequality
\[ \sum_{1 \leq i \leq s} v_p(\varphi_{i}) - \frac{1}{e}\sum_{1 \leq i \leq s} \sum_{\sigma \in \Upsilon} \kappa_{\sigma,i} \geq \frac{1}{e} \sum_{1 \leq i \leq s} \kappa_{\tau,\theta_{\tau}(i)} - \kappa_{\tau,i}  \]
Let $r$ be minimal such that $\theta_{\tau}(s) \neq s$ (if no such $s$ exists, we are done).
In that case, we necessarily have $\theta_{\tau}(s) \geq s+1$, and the previous inequality yields
\[ \sum_{1 \leq i \leq s} v_p(\varphi_{i}) - \frac{1}{e}\sum_{1 \leq i \leq s} \sum_{\sigma \in \Upsilon} \kappa_{\sigma,i} \geq \frac{\kappa_{\tau,s+1} - \kappa_{\tau,s}}{e} \]
but the hypothesis implies that the left hand side is less than $\min_j \left( \kappa_{\tau,j+1} - \kappa_{\tau,j} \right)/e$, and we get a contradiction.
\end{proof}

\begin{theo}
\label{theo:mainressympl}
Let $\Pi$ be an irreducible automorphic representation of $\Galg(\A_F)$ having Iwahori invariants at all the places of $F$ above $p$, and having invariants under an open subgroup $U$ of $\Galg(\A_{F,f}^{(p)})$.
Let $N$ be an integer.
There exists an automorphic representation $\Pi'$ of $\Galg(\A_F)$ such that:
\begin{itemize}
\item $\Pi'$ is unramified at the places above $p$, and has invariants under $U$;
\item The restriction of $\rho_{\iota_p,\iota_{\infty}}(\Pi')$ to the decomposition group at any place above $p$ is either irreducible or the sum of an Artin character and an irreducible representation of dimension $2n$;
\item For all $g$ in $G_F$, $\Tr(\rho_{\iota_p,\iota_{\infty}}(\Pi')(g)) \equiv \Tr(\rho_{\iota_p,\iota_{\infty}}(\Pi)(g)) \mod p^N$.
\end{itemize}
\end{theo}
\begin{proof}
We will write $\Pi' \equiv \Pi \mod p^N$ for the last property.

Recall that for $v$ a place of $F$ above $p$, there are elements $\phi_{v,1}, \ldots , \phi_{v,n} \in \Ocal(\Xscr)^{\times}$ such that for any unramified classical point $x \in \Xscr(\Qpbar)$ refining an automorphic representation $\Pi$, the filtered $\varphi$-module associated with the crystalline representation $\rho_{\iota_p,\iota_{\infty}}(\Pi)|_{G_{F_v}}$ has $\varphi^{f_v}$-eigenvalues
\begin{align*}
\left( \phi_{v,-n}(x) q_v^{-n} \prod_{\sigma} \sigma(\varpi_v)^{k_{v,\sigma,-1}}, \ldots,  \phi_{v,-1}(x) q_v^{-1} \prod_{\sigma} \sigma(\varpi_v)^{k_{v,\sigma,-n}}, 1, \right. \\ \left. \phi_{v,1}(x) q_v \prod_{\sigma} \sigma(\varpi_v)^{k_{v,\sigma,n}}, \ldots, \phi_{v,n}(x)q_v^n \prod_{\sigma} \sigma(\varpi_v)^{k_{v,\sigma,1}} \right)
\end{align*}
and $\sigma$-Hodge-Tate weights
\[ k_{v,\sigma,-1}-n, \ldots , k_{v,\sigma,-n}-1, 0, k_{v,\sigma,n}+1, \ldots, k_{v,\sigma,1}+n \]
In the following if $x_b$ or $x_b'$ is a classical point, $k_{v,\sigma,i}^{(b)}$ will be the weights defining $w(x_b)$.
The representation $\Pi$ corresponds to at least one point $x$ of the eigenvariety $\Xscr$ for $\Galg'$ and the idempotent $e_U \otimes e_{G_0}$.
By Proposition \ref{prop:zardens}, and since $G_F$ is compact, there exists a point $x_1 \in \Xscr(E')$ (near $x$, and for some finite extension $E'$ of $E$) corresponding to an unramified, completely refinable automorphic representation $\Pi_1$ and a refinement $\chi$, such that for any $v$,
\[ \frac{2}{e_v}\sum_{i=1}^n \sum_{\sigma} k^{(1)}_{v,\sigma,i} > -v_p\left( \phi_{v,1}(x_1) \ldots \phi_{v,n}(x_1) \right) + 3n(n+1)f_v \]
and $\Pi_1 \equiv \Pi \mod p^N$.
Since  $\Pi_1$ is completely refinable, there is a point $x_1' \in \Xscr(E')$ associated with the representation $\Pi_1$ and the character $\chi^a$, where $a$ is the element of the Weyl group acting as $-\mathrm{Id}$ on the roots.
Specifically, $\Psi_{x_1} |_{\Hcal^S \otimes \Hcal_S \otimes e_{G_0}} = \Psi_{x_1'} |_{\Hcal^S \otimes \Hcal_S \otimes e_{G_0}}$, but
\[\phi_{v,n+1-i}(x_1')=\phi_{v,-n-1+i}(x_1) q_v^{2i+(2n+2)} \prod_{\sigma} \sigma(\varpi_v)^{-2k^{(1)}_{v,\sigma,i}} \]
for $i=1,\ldots,n$, and all places $v$.
There exists a point $x_2 \in \Xscr(E')$ (near $x_1'$, and up to enlarging $E'$) corresponding to an unramified, completely refinable automorphic representation $\Pi_2$ and a refinement, such that for any $v$ and any $j<0$,
\[\frac{1}{e_v} \sum_{\sigma}k^{(2)}_{v,\sigma,n+j}-k^{(2)}_{v,\sigma,n+j+1}>-v_p(\phi_{v,-j+1}(x_2))-f_v\]
and $\Pi_2 \equiv \Pi_1 \equiv \Pi \mod p^N$.
Like before, since $\Pi_2$ is completely refinable, there is a point $x_2' \in \Xscr(E')$ such that $\Psi_{x_2} |_{\Hcal^S \otimes \Hcal_S \otimes e_{G_0}} = \Psi_{x_2'} |_{\Hcal^S \otimes \Hcal_S \otimes e_{G_0}}$, and
\begin{align*}
& \phi_{v,n}(x_2')=\phi_{v,1}(x_2) q_v^{1-n} \prod_{\sigma} \sigma(\varpi_v)^{k^{(2)}_{v,\sigma,n}-k^{(2)}_{v,\sigma,1}}  \\
& \phi_{v,i}(x_2')=\phi_{v,i+1}(x_2) q_v \prod_{\sigma} \sigma(\varpi_v)^{k^{(2)}_{v,\sigma,n-i}-k^{(2)}_{v,\sigma,n-i+1}}  \text{ for $i=1,\ldots,n-1$.}
\end{align*}
Here we used the element of the Weyl group corresponding (at each $v$) to the permutation
\[
\begin{pmatrix}
-n & -n+1 & \ldots & -2 & -1 & 1 & \ldots & n \\
-n+1 & -n+2 & \ldots & -1 & -n & n & \ldots & n-1
\end{pmatrix} .
\]
Again, we can choose a point $x_3 \in \Xscr(E')$ (near $x_1'$, and up to enlarging $E'$) corresponding to an unramified automorphic representation $\Pi_3$ and a refinement, such that for any $v$ and any $\tau \in \Upsilon$,
\[\begin{split} \frac{1}{e_v (2n+1)} \min \left\{k^{(3)}_{v,\tau,1}-k^{(3)}_{v,\tau,2}, \ldots , k^{(3)}_{v,\tau,n-1}-k^{(3)}_{v,\tau,n}, k^{(3)}_{v,\tau,n} \right\} > \\
\max\left\{ 0, |v_p(\phi_{v,\tau,1}(x_3))|, \ldots, |v_p(\phi_{v,\tau,n}(x_3))| \right\} \end{split}\]
and $\Pi_3 \equiv \Pi \mod p^N$.
We show that $\Pi_3$ has the desired properties.
First we apply the previous lemma to the local Galois representations associated with $\Pi_3$, at the places above $p$, which are crystalline.
Since the differences $v_p(\varphi_i) - \frac{1}{e} \sum_{\sigma \in \Upsilon} \kappa_{\sigma,i}$ in the hypotheses of the lemma are equal in our case to
\[ -v_p(\phi_{v,n}(x_3)), \ldots, -v_p(\phi_{v,1}(x_3)), 0, v_p(\phi_{v,1}(x_3)), \ldots , v_p(\phi_{v,n}(x_3)), \] the hypotheses of the lemma are satisfied for all $\tau \in \Upsilon$.
Thus if $\rho_{\iota_p,\iota_{\infty}}(\pi_3)|_{G_{F_v}}$ is not irreducible, there is a subset $\emptyset \subsetneq I \subsetneq \left\{ -n , \ldots, n \right\}$ such that if $i_1 < \ldots < i_r$ are the elements of $I$ and $j_1 < \ldots < j_{2n+1-r}$ those of $J = \left\{ -n , \ldots, n \right\} \setminus I$,
\begin{align*}
& v_p(\phi_{v,i_1}(x_3)) \geq 0 \\
& v_p(\phi_{v,i_1}(x_3)) + v_p(\phi_{v,i_2}(x_3)) \geq 0 \\
& \vdots \\
& v_p(\phi_{v,i_1}(x_3)) + \ldots + v_p(\phi_{v,i_r}(x_3)) = 0 \\
& v_p(\phi_{v,j_1}(x_3)) \geq 0 \\
& v_p(\phi_{v,j_1}(x_3)) + v_p(\phi_{v,j_2}(x_3)) \geq 0 \\
& \vdots \\
& v_p(\phi_{v,j_1}(x_3)) + \ldots + v_p(\phi_{v,j_{2n+1-r}}(x_3)) = 0
\end{align*}
by the admissibility of the corresponding filtered $\varphi$-modules.
For all $i$, $v_p(\phi_{v,i}(x_2'))=v_p(\phi_{v,i}(x_3))$, so all these conditions hold also at $x_2'$.
Up to exchanging $I$ and $J$, we can assume that $i_1 = -n$.
If $j_1<0$,
\[ v_p(\phi_{v,j_1}(x_2')) = - v_p(\phi_{v,-j_1}(x_2')) = -v_p(\phi_{v,-j_1+1}(x_2))-f_v - \frac{1}{e_v} \sum_{\sigma} k^{(2)}_{v,\sigma,n+j_1}-k^{(2)}_{v,\sigma,n+j_1+1} \]
and $x_2$ was chosen to ensure that this quantity is negative, so we are facing a contradiction.
Thus $J$ has only nonnegative elements, and $\left\{ -n, \ldots, -1 \right\} \subset I$.
If we do not assume that $i_1=-n$, we have in general that $\left\{ -n, \ldots, -1 \right\}$ is contained in $I$ or $J$.
Similarly, suppose $i_r=n$.
If $j_{2n+1-r}>0$,
\begin{align*}
v_p(\phi_{v,j_{2n+1-r}}(x_2')) &= v_p(\phi_{v,j_{2n+1-r}}(x_2')) \\
&= v_p(\phi_{v,j_{2n+1-r}}(x_2))+f_v + \frac{1}{e_v} \sum_{\sigma} k^{(2)}_{v,\sigma,n-j_{2n+1-r}}-k^{(2)}_{v,\sigma,n-j_{2n+1-r}+1}
\end{align*}
is positive, another contradiction.
Therefore $\left\{ 1, \ldots, n \right\}$ is contained in $I$ or $J$.

Assume for example that $\left\{-n, \ldots, -1 \right\} \subset I$ and $\left\{ 1, \ldots, n \right\} \subset J$.
In that case
\begin{eqnarray*}
v_p(\phi_{v,j_1}(x_3) \ldots \phi_{v,j_{2n+1-r}}(x_3)) &=& v_p(\phi_{v,1}(x_2)\ldots \phi_{v,n}(x_2)) \\
&= & v_p(\phi_{v,1}(x_1')\ldots \phi_{v,n}(x_1')) \\
&=  &-v_p(\phi_{v,1}(x_1)\ldots \phi_{v,n}(x_1))+3n(n+1)f_v \\
& & -\frac{2}{e_v}\sum_{i=1}^n \sum_{\sigma} k^{(1)}_{v,\sigma,i}
\end{eqnarray*}
is negative, which is yet another contradiction.

As a consequence, we can conclude that $I$ or $J$ is equal to $\{0\}$, and this shows that at each place $v$ of $F$ above $p$, the semisimplification of $\rho_{\iota_p,\iota_{\infty}}(\Pi_3)|_{G_{F_v}}$ is either irreducible or the sum of an Artin character and an irreducible representation of dimension $2n$.
Consequently $\Pi_3$ has the required properties.
\end{proof}

\section{Similar results for even orthogonal groups}
\label{section:orth}

In this section we explain (very) briefly how the same method as in the previous sections applies to orthogonal groups.

Let $F$ be a totally real number field of even degree over $\Q$.
Then $F$ has an even number of $2$-adic places of odd degree over $\Q_2$, and since these are the only finite places of $F$ at which $(-1,-1)_v=-1$ (where $(\cdot,\cdot)_v$ denotes the Hilbert symbol), $\prod_{v} (-1,-1)_v = 1$ where the product ranges over the finite places of $F$.
Consequently, there is a unique quadratic form on $F^4$ which is positive definite at the real places of $F$, and split (isomorphic to $(x,y,z,t) \mapsto xy+zt$) at the finite places.
It has Hasse invariant $(-1,-1)_v$ at each finite place $v$ of $F$, and its discriminant is $1$.
As a consequence, for any integer $n \geq 1$, there is a connected reductive group $\Galg$ over $F$ which is compact (and connected) at the real places (isomorphic to $\SO_{4n}/\R$) and split at all the finite places (isomorphic to the split $\SO_{4n}$).
As before, we let $\Galg' = \Res^{F}_{\Q} \Galg$.
The proofs of the existence and properties of the attached eigenvariety $\Xscr \rightarrow \Wscr$ are identical to the symplectic case.
We could not find a result as precise as Theorem \ref{theo:irrind} in the literature, however by \cite[Proposition 3.5]{Cassel2} unramified principal series are irreducible on an explicit Zariski-open subset of the unramified characters.
Specifically, if $\SO_{4n}(F_v) = \left\{ M \in \mathrm{M}_{4n}(F_v)\ |\ ^tMJ_{4n}M=J_{4n} \right\}$, 
\[T = \left\{ \begin{pmatrix} 
x_1 &  &  &  &  &  \\ 
 & \ddots &  &  &  &  \\ 
 &  & x_n &  &  &  \\ 
 &  &  & x_n^{-1} &  &  \\ 
 &  &  &  & \ddots &  \\ 
 &  &  &  &  & x_1^{-1}
\end{pmatrix} \ \middle|\  x_i \in F_v^{\times} \right\}\] and $P$ is  any parabolic subgroup containing $T$, then for an unramified character $\chi=(\chi_1,\ldots,\chi_n)$ of $T$ ($\chi_i$ is a character of the variable $x_i$), $\Ind_P^{\SO_{4n}(F_v)} \chi$ is irreducible if $\chi_i(\varpi_v)^2 \neq 1$ for all $i$ and $\chi_i(\varpi_v)\chi_j(\varpi_v)^{\pm 1} \neq 1,q_v,q_v^{-1}$ for all $i<j$.
Note that this is not an equivalence.

The existence of Galois representations $\rho_{\iota_p,\iota_{\infty}}(\Pi)$ attached to automorphic representations $\Pi$ of $\Galg(\A_F)$ is identical to Assumption \ref{prop:transfersp}.
We now state the main result for orthogonal groups.
\begin{theo}
\label{theo:mainresorth}
Let $\Pi$ be an irreducible automorphic representation of $\Galg(\A_F)$ having Iwahori invariants at all the places of $F$ above $p$, and having invariants under an open subgroup $U$ of $\Galg(\A_{F,f}^{(p)})$.
Let $N$ be an integer.
There exists an automorphic representation $\Pi'$ of $\Galg(\A_F)$ such that:
\begin{itemize}
\item $\Pi'$ is unramified at the places above $p$, and has invariants under $U$;
\item The restriction of $\rho_{\iota_p,\iota_{\infty}}(\Pi')$ to the decomposition group at any place above $p$ is irreducible;
\item For all $g$ in $G_F$, $\Tr(\rho_{\iota_p,\iota_{\infty}}(\Pi')(g)) \equiv \Tr(\rho_{\iota_p,\iota_{\infty}}(\Pi)(g)) \mod p^N$.
\end{itemize}
\end{theo}
\begin{proof}
The proof is nearly identical to that of Theorem \ref{theo:mainressympl}.
In the orthogonal case the Weyl group is a bit smaller: it is the semi-direct product of $S_{2n}$ and a hyperplane of $\left(\Z/2\Z \right)^{2n}$.
Alternatively, it is the group of permutations $w$ of $\left\{ -2n, \ldots, -1, 1, \ldots, 2n \right\}$ such that $w(-i)=-w(i)$ for all $i$ and $\prod_{i=1}^{2n} w(i) > 0$.
The two elements of the Weyl group used in the proof of Theorem \ref{theo:mainressympl} have natural counterparts in this Weyl group.
The only difference lies in the fact that there is no Hodge-Tate weight equal to $0$ in the orthogonal case, hence the simpler conclusion ``$\rho_{\iota_p,\iota_{\infty}}(\Pi')|_{G_{F_v}}$ is irreducible for $v | p$''.
\end{proof}

\section{The image of complex conjugation: relaxing hypotheses in Taylor's theorem}

Let us apply the previous results to the determination of the image of the complex conjugations under the $p$-adic Galois representations associated with regular, algebraic, essentially self-dual, cuspidal automorphic representations of $\GL_n(\A_F)$, $F$ totally real.
Recall that these representations are constructed by ``patching'' representations of Galois groups of CM extensions of $F$, on Shimura varieties for unitary groups.
The complex conjugations are lost when we restrict to CM fields.
In \cite{Tay}, Taylor proves that the image of any complex conjugation is given by (the ``discrete'' part of) the local Langlands parameter at the corresponding real place, assuming $n$ is odd and the Galois representation is irreducible, by constructing the complex conjugation on the Shimura datum.
Of course the Galois representation associated with a cuspidal representation of $\GL_n$ is conjectured to be irreducible, but unfortunately this is (at the time of writing) still out of reach in the general case (however, see \cite{CalGee} for $n \leq 5$ and \cite[Theorem D]{BLGGT} for a ``density one'' result).

The results of the first part of this paper allow to remove the irreducibility hypothesis in Taylor's theorem, and to extend it to some (``half'') cases of even $n$, using Arthur's endoscopic transfer.
Unfortunately some even-dimensional cases are out of reach using this method, because odd-dimensional essentially self-dual cuspidal representations are (up to a twist) self-dual, whereas some even-dimensional ones are not.

Since the proof is not direct, let us outline the strategy.
First we deduce the even-dimensional self-dual case from Taylor's theorem by adding a cuspidal self-dual (with appropriate weights) representation of $\GL_3$, we get an automorphic self-dual representation of $GL_{2n+3}$ which (up to base change) can be ``transferred'' to a discrete representation of the symplectic group in dimension $2n$.
Since the associated Galois representation contains no Artin character, it can be deformed irreducibly, and Taylor's theorem applies.
Then the general odd-dimensional case is deduced from the even-dimensional one, by essentially the same method, using the eigenvariety for orthogonal groups.

Finally we prove a supplementary, non-regular case, thanks to the fact that discrete Langlands parameters for the group $\SO_{2n}/\R$ are not always discrete when seen as parameters for $\GL_{2n}$, i.e.\ can correspond to a non-regular representation of $\GL_{2n}/\R$.

\subsection{Regular, L-algebraic, self-dual, cuspidal representations of $\GL_{2n}(\A_F)$ having Iwahori-invariants}

In this subsection $\Galg$ will denote the symplectic group in dimension $2n+2$ defined in section \ref{section:sympl}.

The following is due to C.\ Moeglin and J.-L.\ Waldspurger.
\begin{lemm}
\label{lemm:iwahoritransfersp}
Let $K$ be a finite extension of $\Qp$.
Let $\phi : W_K \times \mathrm{SU}(2) \rightarrow \SO_{2n+3}(\C)$ be a Langlands parameter (equivalently, a generic Arthur parameter).
Assume that the subgroup $I \times \{1\}$ ($I$ being the inertia subgroup of $W_K$) is contained in the kernel of $\phi$.

Then the A-packet associated with $\phi$ contains a representation having a non-zero vector fixed under the Iwahori subgroup of $\Sp_{2n+2}(K)$.
\end{lemm}
\begin{proof}
Let $\{\Pi_1, \ldots, \Pi_k \}$ denote the A-packet.
Since Arthur's construction of the $\Pi_i$'s is inductive for parameters trivial on the supplementary $\SL_2(\C)$, and subquotients of parabolic inductions of representations having Iwahori-invariants have too, it is enough to prove the result when $\phi$ is discrete.
Let $\tau$ be the irreducible smooth representation of $\GL_{2n+3}(K)$ having parameter $\phi$, then
$\tau \simeq \Ind_{L}^{\GL_{2n+3}} \sigma$, where $\sigma$ is the tensor product of (square-integrable) Steinberg representations $\St(\chi_i,n_i)$ of $\GL_{n_i}(K)$ ($i \in  \{1,\ldots,r\}$), $\chi_i$ are unramified, auto-dual characters of $K^{\times}$ (thus $\chi_i=1$ or $(-1)^{v(\cdot)}$), and the couples $(\chi_i,n_i)$ are distinct.
Here $L$ denotes the standard parabolic associated with the decomposition $2n+3=\sum_i n_i$.
Since $\phi$ is self-dual, $\tau$ can be extended (not uniquely, but this will not matter for our purpose) to a representation of $\widetilde{\GL}^+_{2n+3} = \GL_{2n+3} \rtimes \{1,\theta\}$, where
\[ \theta (g) = \begin{pmatrix}
 &  &  & 1 \\ 
 &  & -1 &  \\ 
 & \iddots &  &  \\ 
1 &  &  & 
\end{pmatrix} {}^tg^{-1} \begin{pmatrix}
 &  &  & 1 \\ 
 &  & -1 &  \\ 
 & \iddots &  &  \\ 
1 &  &  & 
\end{pmatrix} \]
Let also $\widetilde{\GL}_{2n+3} = \GL_{2n+3} \rtimes \theta$.

Let $N_0$ be the number of $i$ such that $n_i$ is odd, and for $j \geq 1$ let $N_j$ be the number of $i$ such that $n_i \geq 2 j$.
Then $N_0 + 2 \sum_{j \geq 1} N_j=2n+3$, and if $s$ is maximal such that $N_s > 0$, we let
\[M = \GL_{N_s} \times \ldots \times \GL_{N_1} \times \GL_{N_0} \times \GL_{N_1} \times \ldots \times \GL_{N_s}\]
which is a $\theta$-stable Levi subgroup of $\GL_{2n+3}$, allowing us to define $\widetilde{M}^+$ and $\widetilde{M}$.
Since the standard (block upper triangular) parabolic containing $M$ is also stable under $\theta$, $\tau_M$ is naturally a representation of $\widetilde{M}^+$, denoted by $\tau_{\widetilde{M}}$.
The constituents of the semi-simplification of $\tau_{\widetilde{M}}$ either stay irreducible when restricted to $M$, in which case they are of the form $\sigma_1 \otimes \sigma_0 \otimes \theta(\sigma_1)$ where $\sigma_1$ is a representation of $\GL_{N_s} \times \ldots \times \GL_{N_1}$ and $\sigma_0$ is a representation of $\widetilde{\GL}_{N_0}$; or they are induced from $M$ to $\widetilde{M}^+$, and the restriction of their character to $\widetilde{M}$ is zero.
Since we are precisely interested in that character, we can forget about the second case.
By the geometrical lemma,
\[ \tau_M^{\mathrm{ss}} \simeq \bigoplus_{w \in W^{L,M}} \Ind_{M \cap w(L)}^{M} w \left(\sigma_{L \cap w^{-1} M}\right)  \]
where $W^{L,M}$ is the set of $w \in S_{2n+3}$ such that $w$ is increasing on $I_1=\{1,\ldots,n_1\}$, $I_2=\{n_1+1,\ldots,n_1+n_2\}$, etc.\ and $w^{-1}$ is increasing on $J_{-s}=\{1,\ldots,N_s\}$, $J_{-s+1}=\{N_s+1,\ldots,N_s+N_{s-1}\}$, etc.
Fix the irreducible representation of $\GL_{N_s} \times \ldots \times \GL_{N_1}$
\[ \sigma_1 = \bigotimes_{j=1}^s \Ind_{T_j}^{\GL_{N_j}} \bigotimes_{i\ |\ n_i \geq 2j} \chi_i |\cdot|^{j-\nu_i} \]
where $T_j$ is the standard maximal torus of $\GL_{N_j}$, $\nu_i=\begin{cases} 0 & n_i \mbox{ odd} \\ 1/2 & n_i \mbox{ even} \end{cases}$.

There is a unique $w$ such that $\Ind_{M \cap w(L)}^{M} w \left(\sigma_{L \cap w^{-1} M}\right)$ admits a subquotient of the form $\sigma_1 \otimes \sigma_0 \otimes \theta(\sigma_1)$ as above, moreover $\Ind_{M \cap w(L)}^{M} w \left(\sigma_{L \cap w^{-1} M}\right)$ is irreducible, and
\[ \sigma_1 = \Ind_{T_0}^{\GL_{N_0}} \bigotimes_{i\ |\ n_i \mbox{ odd}} \chi_i \]
Specifically, $w$ maps the first element of $I_i$ in $J_{-\lfloor(n_i+1)/2 \rfloor}$, the second in $J_{-\lfloor(n_i+1)/2 \rfloor}+1$, \ldots, the central element (if $n_i$ is odd) in $J_0$, etc.

Let $M'$ be the parabolic subgroup of $\Sp_{2n+2} \slash K$ corresponding to $M$, i.e.\ 
\[ M' = \GL_{N_s} \times \ldots \times \GL_{N_1} \times \Sp_{N_0-1} \]

By \cite[2.2.6]{Arthur}, $\sum_i \Tr \Pi_i$ is a stable transfer of $\Tr_{\GL_{2n+3}^+} \tau$.
By \cite[Lemme 4.2.1]{MWtransfert} (more accurately, the proof of the lemma),
\[ \sum_i \Tr \left( (\Pi_i)_{M'}^{\mathrm{ss}}[\sigma_1] \right) \]
is a stable transfer of $\Tr \left( \tau_{\widetilde{M}}^{\mathrm{ss}} [\sigma_1] \right)$ (where $\cdot[\cdot]$ denotes the isotypical component on the factor $\GL_{N_s} \times \ldots \times \GL_{N_1}$).

Since $\tau_{\widetilde{M}}^{\mathrm{ss}} [\sigma_1] = \sigma_1 \otimes \sigma_0 \otimes \theta (\sigma_1)$, the stable transfer of $\Tr \left( \tau_{\widetilde{M}}^{\mathrm{ss}} [\sigma_1] \right)$ is equal to the product of $\Tr (\sigma_1)$ and $\sum_{l} \Tr \Pi'_l$ where the $\Pi'_l$ are the elements of the A-packet associated with the parameter
\[ \bigoplus_{i\ |\ n_i \mbox{ odd}} \chi_i \]
At least one representation $\Pi'_l$ is unramified for some hyperspecial compact subgroup of $\Sp_{N_0-1}(K)$, and so a Jacquet module of a $\Pi_i$ contains a nonzero vector fixed by an Iwahori subgroup.
This proves that at least one of the $\Pi_i$ has Iwahori-invariants.
\end{proof}

\begin{assu}
\label{lemm:transfergl2sp}
Let $F_0$ be a totally real field, and let $\pi$ be a regular, L-algebraic, self-dual, cuspidal (RLASDC) representation of $\GL_{2n}(\A_{F_0})$.
Assume that for any place $v|p$ of $F_0$, $\pi_v$ has vectors fixed under an Iwahori subgroup of $\GL_{2n}(\A_{F_{0,v}})$.
Then there exists a RLASDC representation $\pi_0$ of $\GL_3(\A_{F_0})$, a totally real extension $F/F_0$ which is trivial, quadratic or quartic, and an automorphic representation $\Pi$ of $\Galg(\A_F)$ such that
\begin{enumerate}
\item For any place $v|p$ of $F_0$, $\pi_{0,v}$ is unramified.
\item $\BC_{F/F_0}(\pi)$ and $\BC_{F/F_0}(\pi_0)$ remain cuspidal.
\item For any place $v$ of $F$ above $p$, $\Pi_v$ has invariants under the action of the Iwahori subgroup $G_0$ of $\Galg(F_v)$.
\item For any finite place $v$ of $F$ such that $\BC_{F/F_0}(\pi)_v$ and $\BC_{F/F_0}(\pi_0)_v$ are unramified, $\Pi_v$ is unramified, and via the inclusion $\SO_{2n+3}(\C) \hookrightarrow \GL_{2n+3}(\C)$, the Satake parameter of $\Pi_v$ is equal to the direct sum of those of $\BC_{F/F_0}(\pi)_v$ and $\BC_{F/F_0}(\pi_0)_v$.
\end{enumerate}
\end{assu}

Let us comment briefly on the proof to come.
First we construct $\pi_0$.
Let $\delta$ be a cuspidal automorphic representation of $\mathrm{PGL}_2/F_0$ which is unramified at the $p$-adic places, Steinberg at the $\ell$-adic places for some arbitrary prime $\ell \neq p$, and whose local langlands parameters at the real places are of the form $\Ind_{W_{\C}}^{W_{\R}} \left( z \mapsto (z/\bar{z})^a \right)$ where $a$ is a half-integer big enough with respects to their analogues appearing in the local Langlands parameters of $\pi$.
Such a representation exists thanks to \cite[Theorem 1B]{Clo2}.
Let $\pi_0$ be the automorphic representation of $\GL_3/F_0$ obtained by fonctoriality from $\delta$ through the adjoint representation of $\widehat{\mathrm{PGL}_2} (\C) = \SL_2(\C)$ on its Lie algebra.
The representation $\pi_0$ exists and is cuspidal by \cite[Theorem 9.3]{GelJac}.
The condition at the $\ell$-adic places ensures that no nontrivial twist of $\delta$ (seen as a representation of $\GL_2/F_0$) is isomorphic to $\delta$, and the cuspidality of $\pi_0$ follows.
We can twist $\pi_0$ by the central character of $\pi$, to ensure that $\pi \oplus \pi_0$ has trivial central character.
Clearly $\pi_0$ is a RLASDC representation of $\GL_3/F_0$.

Note that for $\BC_{F/F_0}(\pi)$ and $\BC_{F/F_0}(\pi_0)$ to remain cuspidal, it is enough for $F/F_0$ to be totally ramified above a finite place of $F_0$ at which $\pi$ and $\pi_0$ are unramified.
To begin with one can choose such a quadratic extension of $F_0$, in order to define $\Galg$.
The automorphic representation $\Psi := \BC_{F/F_0}(\pi) \oplus \BC_{F/F_0}(\pi_0)$ can be seen as a global, orthogonal parameter.
This determines a global packet $P_{\Psi}$ of representations of $\Galg(\A_F)$, and Arthur's results shall attach to each $\Pi \in P_{\Psi}$ a character of $S_{\Psi} \simeq \Z/2\Z$, and characterize the automorphic $\Pi$'s as the ones whose character is trivial.
We can choose the components $\Pi_v$ at the finite places of $F$ not lying above $p$ to be associated with a trivial character of $S_{\Psi_v}$, and taking a quadratic extension split above the $p$-adic and real places of $F$ (at which $\Pi_v$ is imposed) allows to ``double'' the contribution of the characters, thus yielding a trivial global character.

\begin{prop}
\label{prop:iwahorisympl}
Let $F$ be a totally real field, and let $\pi$ be a regular, L-algebraic, self-dual, cuspidal representation of $\GL_{2n}(\A_F)$.
Suppose that for any place $v$ of $F$ above $p$, $\pi_v$ has invariants under an Iwahori subgroup.
Then for any complex conjugation $c \in G_F$, $\Tr(\rho_{\iota_p,\iota_{\infty}}(\pi)(c))=0$.
\end{prop}
\begin{proof}
By the previous assumption, up to a (solvable) base change to a totally real extension (which only restricts the Galois representation to this totally real field, so that we get even more complex conjugations), we can take a RLASDC representation $\pi_0$ of $\GL_3(\A_F)$ and transfer $\pi \oplus \pi_0$ to an automorphic representation $\Pi$ of $\Galg(\A_F)$.
The representation $\Pi$ defines (at least) one point $x$ of the eigenvariety $\Xscr$ defined by $\Galg$ (and by an open subgroup $U$ of $\Galg(\A_{F,f}^{(p)})$).
Of course, by the \v{C}ebotarev density theorem and the compatibility of the transfer at the unramified places, the representation associated with $\Pi$ is equal to $\rho_{\iota_p,\iota_{\infty}}(\pi) \oplus \rho_{\iota_p,\iota_{\infty}}(\pi)$.
Since the Hodge-Tate weights of $\rho_{\iota_p,\iota_{\infty}}(\pi)|_{G_{F_v}}$ are non-zero for any place $v|p$, $\rho_{\iota_p,\iota_{\infty}}(\pi)$ does not contain an Artin character.
By \cite{BlaRog}, $\rho_{\iota_p,\iota_{\infty}}(\pi_0)$ is irreducible and thus does not contain any character.
There are only finitely many Artin characters taking values in $\{\pm 1\}$ and unramified at all the finite places at which $\Pi$ is unramified.
For any such character $\eta$, the pseudocharacter $T$ on the eigenvariety is such that $T_x - \eta$ is not a pseudocharacter, hence we can find $g_{\eta,1},\ldots,g_{\eta,2n+3}$ such that
\[ t_{\eta} := \sum_{\sigma \in S_{2n+3}} (T_x-\eta)_{\sigma}(g_{\eta,1},\ldots,g_{\eta,2n+3}) \neq 0 \]
Let us choose $N$ greater than all the $v_p(t_{\eta})$ and such that $p^N > 2n+4$.
Let $\Pi'$ be an automorphic representation of $\Galg(\A_F)$ satisfying the requirements of Theorem \ref{theo:mainressympl} for this choice of $N$.
Then the $\Tr(\rho_{\iota_p,\iota_{\infty}}(\Pi'))-\eta$ are not pseudocharacters, thus $\rho_{\iota_p,\iota_{\infty}}(\Pi')$ does not contain an Artin character and by Theorem \ref{theo:mainressympl} it is irreducible.
This Galois representation is (by construction in the proof  of Corollary \ref{coro:galrepsp}) the direct sum of representations associated with cuspidal representations.
Since it is irreducible, there is only one of them, and it has the property that its associated Galois representations is irreducible, so that the theorem of \cite{Tay} can be applied: for any complex conjugation $c \in G_F$, $\Tr(\rho_{\iota_p,\iota_{\infty}}(\Pi')(c))=\pm 1$.
Since $\det \rho_{\iota_p,\iota_{\infty}}(\Pi') = 1$, $\Tr(\rho_{\iota_p,\iota_{\infty}}(\Pi')(c))= (-1)^{n+1}$.

As $p^N > 2n+4$ and $|\Tr(\rho_{\iota_p,\iota_{\infty}}(\Pi)(c)) - \Tr(\rho_{\iota_p,\iota_{\infty}}(\Pi')(c))| \leq 2n+4$, we can conclude that $\Tr(\rho_{\iota_p,\iota_{\infty}}(\Pi)(c))=(-1)^{n+1}$, and hence that $\Tr(\rho_{\iota_p,\iota_{\infty}}(\pi)(c))+\Tr(\rho_{\iota_p,\iota_{\infty}}(\pi_0)(c))=(-1)^{n+1}$.
We also know that $\det \rho_{\iota_p,\iota_{\infty}}(\pi_0) = \det \rho_{\iota_p,\iota_{\infty}}(\pi)(c) = (-1)^n$, and that $\Tr(\rho_{\iota_p,\iota_{\infty}}(\pi_0)(c))= \pm 1$ by Taylor's theorem, from which we can conclude that $\Tr(\rho_{\iota_p,\iota_{\infty}}(\pi_0)(c))=(-1)^{n+1}$.
Thus $\Tr(\rho_{\iota_p,\iota_{\infty}}(\pi)(c))=0$.
\end{proof}

\subsection{Regular, L-algebraic, self-dual, cuspidal representations of $\GL_{2n+1}(\A_F)$ having Iwahori-invariants}

In this subsection, $\Galg$ is the orthogonal reductive group defined in section \ref{section:orth}, of dimension $2n+2$ if $n$ is odd, $2n+4$ if $n$ is even.

\begin{lemm}
\label{lemm:iwahoritransferso}
Let $K$ be a finite extension of $\Qp$.
Let $\phi : W_K \times \mathrm{SU}(2) \rightarrow \SO_{2m}(\C)$ be a Langlands parameter.
Assume that the subgroup $I \times \{1\}$ ($I$ being the inertia subgroup of $W_K$) is contained in the kernel of $\phi$.

Then the packet of representations of the split group $\SO_{2m}(K)$ associated with $\phi$ by Arthur contains a representation having a non-zero vector fixed under the Iwahori subgroup.
\end{lemm}
\begin{proof}
Of course this result is very similar to \ref{lemm:iwahoritransfersp}.
However Moeglin and Waldspurger have not put their lemma in writing in this case, and the transfer factors are no longer trivial, so that one needs to modify the definition of ``stable transfer''.
For this one needs to use the transfer factors $\Delta_{\widetilde{\GL}_{2m},\SO_{2m}}(\cdot ,\cdot)$ defined in \cite{KS}.
They depend in general on the choice of an inner class of inner twistings \cite[1.2]{KS} (in our case an inner class of isomorphisms between $\GL_{2m}/K$ and its quasi-split inner form defined over $\bar{K}$, which we just take to be the identity), and a Whittaker datum of the quasi-split inner form.
Arthur chooses the standard splitting of $\GL_{2m}$ and an arbitrary character $K \rightarrow \C^{\times}$, but this will not matter to us since both $\GL_{2m}$ and $\SO_{2m}$ are \emph{split}, so that the factor $\langle z_J, s_J \rangle$ of \cite[4.2]{KS} (by which the transfer factors are multiplied when another splitting is chosen) is trivial.
Indeed to compute this factor we can choose the split torus $T_H$ of $\SO_{2m}/K$, which is a norm group (see \cite[Lemma 3.3B]{KS}) for the split torus $T$ of $\GL_{2m}/K$, and thus, using the notations of \cite[4.2]{KS}, $T^x$ is split and $H^1(K,T^x)$ is trivial, so that $z'=1$ ($z_J$ is the image of $z'$ in $H^1(K,J)$, so that it is trivial).
Since both groups are split the $\epsilon$-factor of \cite[5.3]{KS} is also trivial, so the transfer factors are canonical.

Let $H = \SO_{2m}(K)$, $\widetilde{\tau}$ the representation of $\widetilde{\GL}^+_{2m}$ associated with $\phi$, and $\tau^H$ the sum of the elements of the packet associated by Arthur with $\phi$.
This packet is in general bigger than the conjectural A-packet, because it corresponds to the $O_{2m}(\C)$-conjugacy class of the parameter $\phi$, and thus it can be the disjoint union of two A-packets.
Arthur shows (\cite[8.3]{Arthur}) that the following character identity holds:
\begin{eqnarray} \Theta_{\tau^H}(h) \Delta(h,g) = \Theta_{\widetilde{\tau}}(g) \label{equ:charident} \end{eqnarray}
whenever the stable $O_{2m}(K)$-conjugacy class $h$ in $\SO_{2m}(K)$ is the (in this case, there is at most one) norm of the conjugacy class $g$ in $\widetilde{\GL}_{2m}(K)$, both assumed to be strongly $\widetilde{\GL}^+_{2m}$-regular.
Here the terms $\Delta_{\mathrm{IV}}$ are left out of the product defining the transfer factors $\Delta$, as in \cite{MWtransfert}.
This is the natural generalization of the notion of ``stable transfer'' of \cite{MWtransfert}.

Let
\[M = \GL_{N_s} \times \ldots \times \GL_{N_1} \times \GL_{N_0} \times \GL_{N_1} \times \ldots \times \GL_{N_s}\]
be a $\theta$-stable Levi subgroup of $\GL_{2m}$, and $M' = \GL_{N_s} \times \ldots \times \GL_{N_1} \times \SO_{N_0}$ the corresponding parabolic subgroup of $\SO_{2m}$.
To mimic the proof of \ref{lemm:iwahoritransfersp}, we only need to show that $\Tr \left( \tau^H_{M'} \right)$ is a stable transfer of $\Tr_{\widetilde{M}} \left(\tau_{\widetilde{M}}\right)$, where ``stable transfer'' has the above meaning, that is the character identity \ref{equ:charident} involving transfer factors.
Note that $\widetilde{M}^+$ has a factor $\GL_{m-N_0/2} \times \GL_{m-N_0/2}$ together with the automorphism $\theta(a,b)=(\theta(b),\theta(a))$, for which the theory of endoscopy is trivial: $\theta$-conjugacy classes are in bijection with conjugacy classes in $\GL_{m-N_0/2}$ (over $K$ or $\bar{K}$) via $(a,b)\mapsto a \theta(b)$ and the $\theta$-invariant irreducible representations are the ones of the form $\sigma \otimes \theta(\sigma)$.

So we need to check that if $g=(g_1,g_0)$ is a strongly regular $\GL_{2m}(K)$-conjugacy class in $\widetilde{\GL}_{2m}(K)$ determined by a conjugacy class $g_1$ in $\GL_{m-N_0/2}(K)$ and a $\GL_{N_0}(K)$-conjugacy class $g_0$ in $\widetilde{\GL}_{N_0}(K)$, and if $h_0$ is the $O_{2m}(\bar{K})$-conjugacy class in $\SO_{2m}(K)$ corresponding to $g_0$, then
\[ \Delta_{\widetilde{\GL}_{N_0},\SO_{N_0}}(h_0,g_0) = \Delta_{\widetilde{\GL}_{2m},\SO_{2m}}((g_1,h_0),(g_1,g_0)) \]
Fortunately the transfer factors have been computed by Waldspurger in \cite{Wald}.
We recall his notations and formulas.
The conjugacy class $g_1$, being regular enough, is parametrized by a finite set $I_1$, a collection of finite extensions $K_{\pm i}$ of $K$ for $i \in I_1$, and (regular enough, i.e.\ generating $K_{\pm i}$ over $K$) elements $x_{i,1} \in K_{\pm i}$.
As in \cite{Wald}, $g_0$ is parametrized by a finite set $I_0$, finite extensions $K_{\pm i}$ of $K$, $K_{\pm i}$-algebras $K_i$, and $x_i \in K_i$.
Each $K_i$ is either a quadratic field extension of $K_{\pm i}$ or $K_{\pm i} \times K_{\pm i}$, and $x_i$ is determined only modulo $N_{K_i/K_{\pm i}} K_i^{\times}$.
Then $g$ is parametrized by $I=I_1 \sqcup I_0$, with $K_i = K_{\pm i} \times K_{\pm i}$ and $x_i=(x_{i,1},1)$ for $i \in I_1$, and the same data for $I_0$.
Let $\tau_i$ be the non-trivial $K_{\pm i}$-automorphism of $K_i$, and $y_i = -x_i/\tau_i(x_i)$.
Let $I^*$ be the set of $i \in I$ such that $K_i$ is a field (so $I^* \subset I_0$).
For any $i \in I$, let $\Phi_i$ be the set of $K$-morphisms $K_i \rightarrow \bar{K}$, and let $P_I(T) = \prod_{i \in I} \prod_{\phi \in \Phi_i} (T - \phi(y_i))$.
Define $P_{I_0}$ similarly.
For $i \in I^*$ (resp.\ $I_0^*$), let $C_i = x_i^{-1} P_I'(y_i)P_I(-1)y_i^{1-m}(1+y_i)$ (resp.\ $C_{i,0} = x_i^{-1} P_{I_0}'(y_i)P_{I_0}(-1)y_i^{1-m}(1+y_i)$).
We have dropped the factor $\eta$ of \cite[1.10]{Wald}, because as remarked above, the transfer factors do not depend on the chosen splitting.
Observe also that the factors computed by Waldspurger are really the factors $\Delta_0/\Delta_{\mathrm{IV}}$ of \cite[5.3]{KS}, but the $\epsilon$ factor is trivial so they are complete.

Waldspurger shows that
\[ \Delta_{\widetilde{\GL}_{2m},\SO_{2m}}((g_1,h_0),(g_1,g_0)) = \prod_{i \in I^*} \mathrm{sign}_{K_i/K_{\pm i}}(C_i) \]
where $\mathrm{sign}_{K_i/K_{\pm i}}$ is the nontrivial character of $K_{\pm i}^{\times}/N_{K_i/K_{\pm i}} K_i^{\times}$.
We are left to show that $\prod_{i \in I^*} \mathrm{sign}_{K_i/K_{\pm i}}(C_i/C_{i,0})=1$.
\begin{eqnarray*}
C_i/C_{i,0} & = & y_i^{N_0/2-m} \prod_{j \in I_1} \prod_{\phi \in \Phi_j} (y_i - \phi(y_j))(-1-\phi(y_j)) \\
& = & \prod_{j \in I_1} \prod_{\phi \in \Phi_{\pm j}} y_i^{-1} \left(y_i + \phi(x_{j,1})\right)\left(y_i + \phi(x_{j,1})^{-1}\right)\left(\phi(x_{j,1}) - 1\right)\left(\phi(x_{j,1})^{-1} -1\right) \\
& = & (-1)^{m-N_0/2} N_{K_i/K_{\pm i}}\left( \prod_{j \in I_1} \prod_{\phi \in \Phi_{\pm j}} (y_i+\phi(x_{j,1})) (\phi(x_{j,1})^{-1}-1) \right)
\end{eqnarray*}
where $\Phi_{\pm j}$ is the set of $K$-morphisms $K_{\pm j} \rightarrow \bar{K}$.
Thus
\begin{eqnarray*}
\prod_{i \in I^*} \mathrm{sign}_{K_i/K_{\pm i}}(C_i/C_{i,0}) & = & \prod_{i \in I^*} \mathrm{sign}_{K_i/K_{\pm i}}|_{K^{\times}} \left( (-1)^{m-N_0/2} \right) \\
& = & 1
\end{eqnarray*}
since $\prod_{i \in I^*} \mathrm{sign}_{K_i/K_{\pm i}}|_{K^{\times}}$ is easily checked to be equal to the Hilbert symbol with the discriminant of our special orthogonal group, which is $1$ (this is the condition for $g_0$ to have a norm in the special orthogonal group).
\end{proof}

\begin{assu}
\label{lemm:transfergl2so}
Let $F_0$ be a totally real field, and let $\pi$ be a regular, L-algebraic, self-dual, cuspidal representation of $\GL_{2n+1}(\A_{F_0})$.
Assume that for any place $v|p$ of $F_0$, $\pi_v$ has vectors fixed under the Iwahori.
Then there exists a RLASDC representation $\pi_0$ of $\GL_1(\A_{F_0})$ if $n$ is odd (resp. $\GL_3(\A_{F_0})$ if $n$ is even), a totally real extension $F/F_0$ which is trivial or quadratic, and an automorphic representation $\Pi$ of $\Galg(\A_F)$ such that
\begin{enumerate}
\item For any place $v|p$ of $F_0$, $\pi_{0,v}$ is unramified.
\item $\BC_{F/F_0}(\pi)$ and $\BC_{F/F_0}(\pi_0)$ remain cuspidal.
\item For any place $v$ of $F$ above $p$, $\Pi_v$ has invariants under the action of the Iwahori subgroup of $\Galg(F_v)$.
\item For any finite place $v$ of $F$ such that $\BC_{F/F_0}(\pi)_v$ and $\BC_{F/F_0}(\pi_0)_v$ are unramified, $\Pi_v$ is unramified, and via the inclusion $\SO_{2n+2}(\C) \hookrightarrow \GL_{2n+2}(\C)$ (resp. $\SO_{2n+4}(\C) \hookrightarrow \GL_{2n+2}(\C)$), the Satake parameter of $\Pi_v$ is equal to the direct sum of those of $\BC_{F/F_0}(\pi)_v$ and $\BC_{F/F_0}(\pi_0)_v$.
\end{enumerate}
\end{assu}

This is very similar to Assumption \ref{lemm:transfergl2sp}.
In fact in this case the group $S_{\Psi}$ is trivial, which explains why it is enough to take a quadratic extension of $F_0$.
This is necessary to be able to define the group $\Galg$.
The crucial observation is that the local Langlands parameters of $\BC_{F/F_0}(\pi) \oplus \BC_{F/F_0}(\pi_0)$ at the infinite places correspond to parameters for the compact groups $\SO_{2n+2}/\R$ (resp. $\SO_{2n+4}$).
These parameters are of the form
\[ \epsilon^n \oplus \bigoplus_{i=1}^n \Ind_{W_{\R}}^{W_{\C}} \left( z \mapsto (z/\bar{z})^{r_i} \right) \]
($r_1>\ldots>r_n>0$) for $\BC_{F/F_0}(\pi)$, and
\[ \begin{cases}
1 & \mbox{if $n$ is odd} \\
\epsilon \oplus \Ind_{W_{\R}}^{W_{\C}} \left( z \mapsto (z/\bar{z})^{r} \right) & \mbox{if $n$ is even}
\end{cases} \]
so that the direct sum of the two is always of the form
\[ 1 \oplus \epsilon \oplus \bigoplus_{i=1}^{k-1} \Ind_{W_{\R}}^{W_{\C}} \left( z \mapsto (z/\bar{z})^{r_i} \right) \]
for distinct, positive $r_i$.
This is the Langlands parameter corresponding to the representation of $\SO_{2k}(\R)$ having highest weight $\sum_{i=1}^{k} (r_i-(k-i)) e_i$ with $r_k=0$, where the root system consists of the $\pm e_i \pm e_j$ ($i \neq j$) and the simple roots are $e_1 - e_2, \ldots , e_{k-1} - e_k, e_{k-1}+e_k$.

Note that, contrary to the symplectic case, there is one outer automorphism of the even orthogonal group, and so there may be two choices for the Satake parameters of $\Pi_v$, mapping to the same conjugacy class in the general linear group.
Fortunately we only need the existence.

\begin{prop}
\label{prop:iwahoriorth}
Let $F$ be a totally real field, and let $\pi$ be an L-algebraic, self-dual, cuspidal representation of $\GL_{2n+1}(\A_F)$.
Suppose that for any place $v$ of $F$ above $p$, $\pi_v$ has invariants under an Iwahori.
Then for any complex conjugation $c \in G_F$, $\Tr(\rho_{\iota_p,\iota_{\infty}}(\pi)(c))=\pm 1$.
\end{prop}
\begin{proof}
The proof is similar to that of Proposition \ref{prop:iwahorisympl}.
We use the previous assumption to be able to assume (after base change) that there is a representation $\pi_0$ (of $\GL_1(\A_{F})$ if $n$ is odd, $\GL_3(\A_{F})$ if $n$ is even) such that $\pi \oplus \pi_0$ transfers to an automorphic representation $\Pi$ of $\Galg(\A_F)$, with compatibility at the unramified places.
The representation $\Pi$ has Iwahori-invariants at the $p$-adic places of $F$, and thus it defines a point of the eigenvariety $\Xscr$ associated with $\Galg$ (and an idempotent defined by an open subgroup of $\Galg(\A_{F,f}^{(p)})$).
By Theorem \ref{theo:mainresorth}, $\Pi$ is congruent (at all the complex conjugations, and modulo arbirarily big powers of $p$) to another automorphic representation $\Pi'$ of $\Galg$, and $\rho_{\iota_p,\iota_{\infty}}(\Pi')$ is irreducible.
Hence $\rho_{\iota_p,\iota_{\infty}}(\Pi') = \rho_{\iota_p,\iota_{\infty}}(\pi')$ for some RLASDC $\pi'$ of $\GL_{2k}(\A_F)$, which is unramified at all the $p$-adic places of $F$, and we can apply Proposition \ref{prop:iwahorisympl} to $\pi'$.
This proves that $\Tr(\rho_{\iota_p,\iota_{\infty}}(\pi)(c) = - \Tr(\rho_{\iota_p,\iota_{\infty}}(\pi_0)(c) = \pm 1$.
\end{proof}

\subsection{Almost general case}

We will now remove the hypothesis of being Iwahori-spherical at $p$, and allow more general similitude characters, using Arthur and Clozel's base change.

\begin{lemm}
Let $E$ be a number field, $S$ a finite set of (possibly infinite) places of $E$, and for each $v \in S$, let $K^{(v)}$ be a finite abelian extension of $E_v$.
There is an abelian extension $F$ of $E$ such that for any $v \in S$ and any place $w$ of $F$ above $v$, the extension $F_v/E_v$ is isomorphic to $K^{(v)}/E_v$.
\end{lemm}
\begin{proof}
After translation to local and global class field theory, this is a consequence of \cite[Théorème 1]{Chevalley}.
\end{proof}

Before proving the last theorem, we need to reformulate the statement, in order to make the induction argument more natural.
Let $\pi$ be a regular, L-algebraic, cuspidal representation of $\GL_{2n+1}(\A_F)$.
At a real place $v$ of $F$, the Langlands parameter of $\pi_v$ is of the form
\[ \epsilon^e \oplus \bigoplus_i \Ind_{W_{\R}}^{W_{\C}} z \mapsto (z/\bar{z})^{n_i} \]
and according to the recipe given in \cite[Lemma 2.3.2]{BuzGee}, $\rho_{\iota_p,\iota_{\infty}}(\pi)(c_v)$ should be in the same conjugacy class as
\[ \begin{pmatrix}
(-1)^e &  &  &  & &  \\
 & 0 & 1 &  &  &  \\
 & 1 & 0 &  &  &  \\
 & & & \ddots & &  \\
 & & & & 0 & 1 \\
 & & & & 1 & 0
\end{pmatrix} \]
Since it is known that $\det \rho_{\iota_p,\iota_{\infty}}(\pi)(c_v) = (-1)^{e+n}$, $\rho_{\iota_p,\iota_{\infty}}(\pi)(c_v) \sim \mathcal{LL}(\pi_v)(j)$ if and only if $|\Tr \rho_{\iota_p,\iota_{\infty}}(\pi)(c_v)|=1$.
Similarly, in the even-dimensional case, $\rho_{\iota_p,\iota_{\infty}}(\pi)(c_v) \sim \mathcal{LL}(\pi_v)(j)$ if and only if $\Tr \rho_{\iota_p,\iota_{\infty}}(\pi)(c_v)=0$.

\begin{theo}
\label{theo:mainthmcc}
Let $n \geq 2$, $F$ a totally real number field, $\pi$ a regular, L-algebraic, essentially self-dual, cuspidal representation of $\GL_n(\A_F)$, such that $\pi^{\vee} \simeq \left((\eta |\cdot|^q) \circ \det \right) \otimes \pi$, where $\eta$ is an Artin character.
Suppose that one of the following conditions holds
\begin{enumerate}
\item $n$ is odd.
\item $n$ is even, $q$ is even, and $\eta_{\infty}(-1)=1$.
\end{enumerate}
Then for any complex conjugation $c \in G_F$, $|\Tr(\rho_{\iota_p,\iota_{\infty}}(\pi)(c))| \leq 1$.
\end{theo}
\begin{proof}
We can twist $\pi$ by an algebraic character, thus multiplying the similitude character $\eta |\cdot|^q$ by the square of an algebraic character.
If $n$ is odd, this allows to assume $\eta=1,q=0$ (by comparing central characters, we see that $\eta |\cdot|^q$ is a square).
If $n$ is even, we can assume  that $q=0$ (we could also assume that the order of $\eta$ is a power of $2$, but this is not helpful).
The Artin character $\eta$ defines a cyclic, totally real extension $F'/F$.
Since local Galois groups are pro-solvable, the preceding lemma shows that there is a totally real, solvable extension $F''/F'$ such that $\BC_{F''/F}(\pi)$ has Iwahori invariants at all the places of $F''$ above $p$.
In general $\BC_{F''/F}(\pi)$ is not cuspidal, but only induced by cuspidals: $\BC_{F''/F}(\pi) = \pi_1 \boxplus \ldots \boxplus \pi_k$.
However it is self-dual, and the particular form of the Langlands parameters at the infinite places imposes that all $\pi_i$ be self-dual.
We can then apply Propositions \ref{prop:iwahorisympl} and \ref{prop:iwahoriorth} to the $\pi_i$, and conclude by induction that for any complex conjugation $c \in G_F$, the conjugacy class of $\rho_{\iota_p,\iota_{\infty}}(\pi)(c)$ is given by the recipe found in \cite[Lemma 2.3.2]{BuzGee}, that is to say $\left| \Tr \rho_{\iota_p,\iota_{\infty}}(\pi)(c) \right| \leq 1$.
\end{proof}

\begin{rema}
The case $n$ even, $\eta_{\infty}(-1)=(-1)^{q+1}$ is trivial.
The case $n$ even, $q$ odd and $\eta_{\infty}(-1)=-1$ remains open.
\end{rema}

For the sake of clarity, we state the theorem using the more common normalization of C-algebraic representations.

\begin{theo}
\label{theo:lastthm}
Let $n \geq 2$, $F$ a totally real number field, $\pi$ a regular, algebraic, essentially self-dual, cuspidal representation of $\GL_n(\A_F)$, such that $\pi^{\vee} \simeq \eta |\det|^q \pi$, where $\eta$ is an Artin character.
Suppose that one of the following conditions holds
\begin{enumerate}
\item $n$ is odd.
\item $n$ is even, $q$ is odd, and $\eta_{\infty}(-1)=1$.
\end{enumerate}
Then for any complex conjugation $c \in G_F$, $|\Tr(r_{\iota_p,\iota_{\infty}}(\pi)(c))| \leq 1$.
\end{theo}
\begin{proof}
Apply the previous theorem to $\pi |\det|^{(n-1)/2}$.
\end{proof}

\subsection{A supplementary, non-regular case}

In this subsection $\Galg$ is the orthogonal group of section \ref{section:orth}.

\begin{assu}
\label{lemm:transfernonreg}
Let $F_0$ be a totally real field, and let $\pi$ be an L-algebraic, self-dual, cuspidal representation of $\GL_{2n}(\A_{F_0})$.
Assume that for any place $v|p$ of $F_0$, $\pi_v$ has vectors fixed under the Iwahori, and that for any real place $v$ of $F_0$,
\[ \mathcal{LL}(\pi_v) \simeq \bigoplus_{i=1}^{n} \Ind_{W_{\R}}^{W_{\C}} \left( z \mapsto (z/\bar{z})^{r_i} \right) \]
where $r_n>\ldots>r_1 \geq 0$ are integers (note that $\pi$ is not regular if $r_1=0$).
Then there exists a totally real extension $F/F_0$ which is trivial or quadratic, and an automorphic representation $\Pi$ of $\Galg(\A_F)$ such that
\begin{enumerate}
\item $\BC_{F/F_0}(\pi)$ remains cuspidal.
\item For any place $v$ of $F$ above $p$, $\Pi_v$ has invariants under the action of the Iwahori subgroup of $\Galg(F_v)$.
\item For any finite place $v$ of $F$ such that $\BC_{F/F_0}(\pi)_v$ is unramified, $\Pi_v$ is unramified, and via the inclusion $\SO_{2n+2}(\C) \hookrightarrow \GL_{2n}(\C)$, the Satake parameter of $\Pi_v$ is equal to the one of $\BC_{F/F_0}(\pi)_v$.
\end{enumerate}
\end{assu}
Of course this is very similar to Assumptions \ref{lemm:transfergl2sp} and \ref{lemm:transfergl2so}, and as in the latter case the group $S_{\Psi}$ is trivial.

For L-algebraic, self-dual, cuspidal automorphic representations of $\GL_{2n}$ having ``almost regular'' Langlands parameter at the archimedean places as above, the corresponding $p$-adic Galois representation is known to exist by \cite{Goldring}.
Exactly as in the previous subsection, we have the following:
\begin{theo}
Let $n \geq 2$, $F$ a totally real number field, $\pi$ an L-algebraic, essentially self-dual, cuspidal representation of $\GL_{2n}(\A_F)$, such that $\pi^{\vee} \simeq \eta \pi$, where $\eta$ is an Artin character.
Assume that at any real place $v$ of $F$, $\eta_v(-1)=1$ and
\[ \mathcal{LL}(\pi_v) \simeq \bigoplus_{i=1}^{n} \Ind_{W_{\R}}^{W_{\C}} \left( z \mapsto (z/\bar{z})^{r_i} \right) \]
where $r_n>\ldots>r_1 \geq 0$ are integers.
Then for any complex conjugation $c \in G_F$, $\Tr(\rho_{\iota_p,\iota_{\infty}}(\pi)(c)) = 0$.
\end{theo}
\begin{proof}
Identical to that of Theorem \ref{theo:mainthmcc}.
\end{proof}

\newpage

\bibliographystyle{amsalpha}
\bibliography{deform}

\end{document}